\documentclass[12pt,reqno]{amsart}

\usepackage{amssymb,mathrsfs}
\usepackage{enumitem}
\usepackage[colorlinks=true,linkcolor=blue,citecolor=blue,urlcolor=blue]{hyperref}
\usepackage[margin=1in]{geometry}

\numberwithin{equation}{section}
\newtheorem{theorem}{Theorem}[section]
\newtheorem{proposition}[theorem]{Proposition}
\newtheorem{lemma}[theorem]{Lemma}
\newtheorem{remark}[theorem]{Remark}
\theoremstyle{definition}
\newtheorem{setup}[theorem]{Setup}

\newcommand{\Sc}{\operatorname{Scal}}
\newcommand{\Ahat}{\widehat A}
\newcommand{\ch}{\operatorname{ch}}
\newcommand{\Str}{\operatorname{Str}}
\newcommand{\Tr}{\operatorname{Tr}}
\newcommand{\tr}{\operatorname{tr}}
\newcommand{\str}{\operatorname{str}}
\newcommand{\supp}{\operatorname{supp}}
\newcommand{\dist}{\operatorname{dist}}
\newcommand{\Deck}{\operatorname{Deck}}
\newcommand{\cR}{\mathcal R}
\newcommand{\cS}{\mathcal S}
\newcommand{\cW}{\mathscr W}

\title[Nonexistence of complete metrics with positive scalar curvature]
{Nonexistence of complete metrics with uniformly positive scalar curvature}
\author{Tsz-Kiu Aaron Chow}
\address{Department of Mathematics, Hong Kong University of Science and Technology, Hong Kong S.A.R., China}
\email{\href{chowtka@ust.hk}{chowtka@ust.hk}}
\thanks{T.-K. A. C. is supported by the Croucher Foundation Start-up Grant and the HKUST New Faculty Start-up Grant.}

\begin{document}

\begin{abstract}
Let $X$ be a connected oriented even-dimensional manifold whose
universal cover is spin.  We prove that $X$ admits no complete metric
with uniformly positive scalar curvature when either infinite relative
$K$-area or a relative cohomological condition holds along compact sets
escaping to infinity in a fixed open subset with compact complement.
The proof compares twisted Dirac operators on spin covers of compact
manifolds with boundary and combines a heat-kernel argument on
fundamental domains with a long-neck estimate.
\end{abstract}

\maketitle

\section{Introduction}

In \emph{Four lectures on scalar curvature}, Gromov formulates two
obstructions to complete metrics of positive scalar curvature in terms
of relative $K$-area and relative cohomology
\cite[\S4.7]{GromovFourLectures}.  His Theorems~5a and~5b concern data
supported in a fixed compact set
\cite[Theorems~5a and~5b]{GromovFourLectures}.  In Conjecture~6a,
Gromov allows the support to escape to infinity and conjectures the
nonexistence of complete metrics with uniformly positive scalar
curvature \cite[Conjecture~6a]{GromovFourLectures}.  In this paper,
we prove this conjecture.

Let $X^n$ be a connected oriented manifold of even dimension
$n=2k\geq2$, let $X^\circ\subset X$ be open with compact complement,
and let $X_i^\circ\subset X^\circ$ be compact sets which escape every
compact subset of $X$.

\begin{theorem}\label{thm:main}
Suppose that the universal cover of $X$ is spin and that one of the
following alternatives holds.
\begin{enumerate}[leftmargin=2.8em]
\renewcommand{\labelenumi}{\textup{(K)}}
\renewcommand{\theenumi}{\textup{(K)}}
\item\label{alt:K}
The relative $K$-areas of $X^\circ$ with respect to
$X^\circ\setminus X_i^\circ$ are infinite for every $i$.

\renewcommand{\labelenumi}{\textup{(H)}}
\renewcommand{\theenumi}{\textup{(H)}}
\item\label{alt:H}
There are classes
\[
 h_i\in H^2(X^\circ,X^\circ\setminus X_i^\circ;\mathbb R),
 \qquad h_i^k\ne0,
\]
whose pullbacks to the universal cover
$p^\circ:\widetilde X^\circ\to X^\circ$ vanish in
\[
 H^2\!\left(\widetilde X^\circ,
 \widetilde X^\circ\setminus(p^\circ)^{-1}(X_i^\circ);\mathbb R\right).
\]
\end{enumerate}
Then $X$ admits no complete metric satisfying $\Sc\geq\sigma>0$.
\end{theorem}

Related Dirac-operator arguments for complete manifolds have been
developed in several settings.
In the $K$-area setting, Cecchini--Zeidler obtain an obstruction on complete
spin manifolds from almost-flat pairs with nonzero relative index
\cite[Corollary~3.10]{CecchiniZeidlerCallias}, while Zhang proves a
corresponding result for bundles which are trivial with trivial
connection outside a compact set
\cite[Theorem~2.3]{ZhangDeformedDirac}. 
Dirac-operator proofs of
long-neck and band-width estimates were developed by Zeidler and
Cecchini \cite{ZeidlerBandWidth,CecchiniLongNeck} and sharpened by
Cecchini--Zeidler \cite{CecchiniZeidlerComparison}.  A different
approach to such distance estimates, based on quantitative $K$-theory
and higher index theory, was developed by Guo--Xie--Yu
\cite{GuoXieYuQuantitative}.
There are also cohomological versions when only the universal cover is assumed to
be spin.  Gromov proves a closed-manifold result for degree-one and
degree-two classes whose pullbacks to the universal cover vanish and
states a complete-manifold analogue using compactly supported classes
\cite[\S10]{GromovMetricInequalities}.  Zhang treats a
noncompact case arising from a map to $\mathbb{CP}^k$ whose lifted
K\"ahler form is exact
\cite[\S2.2]{ZhangDeformedDirac}.

There are two points which prevent a direct reduction of
Theorem~\ref{thm:main} to these fixed-support results.  First,
$X^\circ$ need not be complete, and the universal cover
$\widetilde X^\circ$ in assumption~\ref{alt:H} need not agree with
the inverse image of $X^\circ$ in the universal cover of $X$.
Second, even when two twisting data agree outside a compact subset
$C$ downstairs, their lifts agree outside $\pi^{-1}(C)$, which is
generally noncompact.  Thus the relative index theorem for operators
which agree outside a compact set
\cite{GromovLawsonComplete} does not directly apply.

We reduce both assumptions to compact domains
$M_i\subset X^\circ$ containing compact subsets
$C_i\Subset M_i^\circ$ with
\[
 \dist(C_i,\partial M_i)\longrightarrow\infty.
\]
Under assumption~\ref{alt:K}, relative $K$-area gives almost-flat
Hermitian bundles on $M_i$, parallel-isomorphic outside $C_i$, with
a nonzero characteristic integral.  Under
assumption~\ref{alt:H}, relative de Rham theory gives a compactly
supported closed two-form $\omega_i$ on $M_i$ and a one-form
$\alpha_i$ on its universal cover satisfying
\[
 d\alpha_i=q_i^*\omega_i,\qquad
 \alpha_i=0\quad\text{on }q_i^{-1}(M_i\setminus C_i),\qquad
 \int_{M_i}\omega_i^k\ne0.
\]
Small multiples of $\alpha_i$ then give line bundle connections with
arbitrarily small curvature.  A further covering construction is
needed because the maps preserving these connections may differ from
those for the trivial connection by constants depending on the
component of $q_i^{-1}(M_i\setminus C_i)$. This is carried out using Bass--Serre theory in Appendix~\ref{app:common-cover}.

The common analytic input is the vanishing theorem proved in
Section~\ref{sec:analytic}.  It compares two twisted Dirac operators
on spin covers of a compact manifold with boundary whose twisting
data agree outside a compact subset of the interior.  If the twisting
curvatures are sufficiently small and this subset is sufficiently far
from the boundary, the associated characteristic integral vanishes.
Applied to the comparison domains above, this contradicts the
nonvanishing obtained from either assumption.

Section~\ref{sec:analytic} proves the vanishing theorem for twisted
Dirac operators.  Sections~\ref{sec:K} and~\ref{sec:H} prove
Theorem~\ref{thm:main} under assumptions~\ref{alt:K} and~\ref{alt:H},
respectively.  Appendix~\ref{app:localization} contains the elementary
constructions of comparison domains and smooth functions, while
Appendix~\ref{app:common-cover} gives the covering and bundle
construction needed in the relative cohomology case.

\textit{Disclosure of AI use.}
The author used ChatGPT (GPT-5.6 Sol) and Claude (Fable 5) as supporting
tools in the preparation of this manuscript, including for literature
searches, exploration of ideas, and symbolic calculations.
In particular, the use of Bass--Serre theory in Appendix~\ref{app:common-cover}
for the covering construction under assumption~\ref{alt:H} was suggested
by ChatGPT.  All mathematical arguments and calculations
were independently verified by the author, who takes full responsibility
for the mathematical content of the paper.


\vspace{0.5cm}

\section{A vanishing theorem for twisted Dirac operators}\label{sec:analytic}

In this section, we prove the vanishing theorem used in
Sections~\ref{sec:K} and~\ref{sec:H}.  We compare two twisted Dirac
operators on spin covers of a compact manifold $M$ with boundary whose
twisting data are identified over $M\setminus C$.  If the twisting
curvatures are sufficiently small and $C$ is sufficiently far from
$\partial M$, the corresponding characteristic integral vanishes.

We use the conventions
\[
 \ch(E)=\Tr\exp\!\left(\frac{\sqrt{-1}\,R^E}{2\pi}\right),\qquad
 \Ahat(TM)=\det{}^{1/2}\!\left(
 \frac{R^{TM}/(4\pi\sqrt{-1})}
 {\sinh\bigl(R^{TM}/(4\pi\sqrt{-1})\bigr)}
 \right).
\]
For a Hermitian bundle $E$ with unitary connection, we use $\|R^E\|_g =\sup_{x,\,|v|=|w|=1} \|R_x^E(v,w)\|_{\mathrm{op}}$ as in \cite[Lemma~5]{BaerHankeKCowaist}.

\subsection{Setup and statement}

\begin{setup}\label{setup:data}
Let $M^n$ be a compact connected oriented Riemannian manifold of even
dimension with $\partial M\ne\varnothing$, and let
$C\Subset M^\circ$ be compact.  Let $\Gamma$ be a countable discrete
group.  We assume the following data throughout this section.
\begin{enumerate}[label=\textup{(S\arabic*)},ref=\textup{(S\arabic*)},
                  leftmargin=3.4em]
\item\label{setup:covering}
For $a=1,0$, let $\pi_a:Z_a\to M$ be a principal
$\Gamma$-bundle with the pulled-back metric and orientation.  The
possibly disconnected manifold $Z_a$ is spin, and $E_a\to Z_a$ is a
finite-rank Hermitian bundle with unitary connection.  Let $S_{Z_a}$
be the spinor bundle of $Z_a$.  We write
\[
 \cS_a=S_{Z_a}\otimes E_a=\cS_a^+\oplus\cS_a^- ,
 \qquad D_a=c_a\circ\nabla^a,
\]
and let $\varepsilon_a$ be the chirality operator on $\cS_a$.

\item\label{setup:lifts}
For every $\gamma\in\Gamma$, there is a smooth unitary lift
$T^a_\gamma:\cS_a\to\cS_a$ of the deck transformation $\gamma:Z_a\to Z_a$.
For every smooth vector field $v$ on $Z_a$ and smooth section $u$,
\begin{equation}\label{eq:T-properties}
 \begin{gathered}
 T^a_\gamma\varepsilon_a=\varepsilon_aT^a_\gamma,\qquad
 T^a_\gamma c_a(v)=c_a(\gamma_*v)T^a_\gamma,\qquad
 \nabla^a_{\gamma_*v}(T^a_\gamma u)
 =T^a_\gamma(\nabla^a_vu).
 \end{gathered}
\end{equation}

\item\label{setup:exterior}
There is a $\Gamma$-equivariant, orientation-preserving isometry
\[
 \Phi:\pi_0^{-1}(M\setminus C)\longrightarrow
 \pi_1^{-1}(M\setminus C),\qquad \pi_1\circ\Phi=\pi_0.
\]
On these regions there are smooth unitary bundle maps
$\widehat\Phi:S_{Z_0}\to S_{Z_1}$ and $U^E:E_0\to E_1$ satisfying
\[
 \operatorname{pr}_{S_{Z_1}}\circ\widehat\Phi
 =\Phi\circ\operatorname{pr}_{S_{Z_0}},\qquad
 \operatorname{pr}_{E_1}\circ U^E
 =\Phi\circ\operatorname{pr}_{E_0}.
\]
Here $\widehat\Phi$ preserves chirality, Clifford multiplication,
and the spin connection, and $U^E$ is parallel.

\item\label{setup:compatibility}
For every $\gamma\in\Gamma$, the map $U=\widehat\Phi\otimes U^E$ satisfies
\begin{equation}\label{eq:exterior-compatibility}
 T^1_\gamma U=UT^0_\gamma.
\end{equation}
\end{enumerate}
\end{setup}

The identities in \eqref{eq:T-properties} give $D_aT^a_\gamma=T^a_\gamma D_a$.
We recall the Weitzenb\"ock formula
\cite[Theorem~II.8.17]{LawsonMichelsohn} 
\begin{equation}\label{eq:SL}
 D_a^2=(\nabla^a)^*\nabla^a+\frac{\Sc}{4}+\cR^{E_a},
\end{equation}
where, for a local orthonormal frame $e_1,\ldots,e_n$,
\begin{equation}\label{eq:twisting-curvature-endomorphism}
 \cR^{E_a}=\sum_{i<j}c_a(e_i)c_a(e_j)R^{E_a}(e_i,e_j).
\end{equation}
We use the norm $\|\cR^{E_a}\|=\sup_{z\in Z_a}
\|\cR^{E_a}_z\|_{\mathrm{op}}$.
For smooth vector fields $v,w$, the identities in
\eqref{eq:T-properties} imply
\[
\begin{aligned}
 T^a_\gamma R^{\cS_a}(v,w)
 &=R^{\cS_a}(\gamma_*v,\gamma_*w)T^a_\gamma,\\
 T^a_\gamma\bigl(R^{S_{Z_a}}(v,w)\otimes1\bigr)
 &=\bigl(R^{S_{Z_a}}(\gamma_*v,\gamma_*w)\otimes1\bigr)T^a_\gamma.
\end{aligned}
\]
Since $R^{\cS_a}=R^{S_{Z_a}}\otimes1+1\otimes R^{E_a}$, the above two identities imply
\begin{equation}\label{eq:twisting-curvature-under-gamma}
 T^a_\gamma\bigl(1\otimes R^{E_a}(v,w)\bigr)
 =\bigl(1\otimes R^{E_a}(\gamma_*v,\gamma_*w)\bigr)T^a_\gamma.
\end{equation}
Taking traces of its powers and using
$\tr_{\cS_a}(1\otimes(R^{E_a})^k)
=\operatorname{rk}(S_{Z_a})\tr_{E_a}((R^{E_a})^k)$ shows that
$\gamma^*\ch(E_a)=\ch(E_a)$.  Thus $\ch(E_a)$ descends to $M$.
The descended forms agree on $M\setminus C$, since $U^E$ is parallel.

\begin{theorem}\label{thm:vanishing}
Assume Setup~\ref{setup:data}. Suppose that
\[
 \Sc_M\geq\sigma>0,\qquad
 \|\cR^{E_a}\|\leq\frac{\sigma}{8}\quad(a=1,0),\qquad
 \dist(C,\partial M)>\frac{2\pi}{\sqrt\sigma}.
\]
Then
\begin{equation}\label{eq:vanishing}
 \int_M\Ahat(TM)\bigl(\ch(E_1)-\ch(E_0)\bigr)=0.
\end{equation}
\end{theorem}

We briefly describe the proof strategy.  We impose a self-adjoint
boundary condition on the pair of Dirac operators and consider a family of
Callias operators $B_r=D+r\Sigma^\psi$, $0\leq r\leq1$,
where $\Sigma^\psi$ is a zeroth-order term depending on a function
$\psi$ to be chosen later.  For a fundamental domain $F$, we consider
\[
 H_r(t)=\Str_F(e^{-tB_r^2}).
\]
We prove that $H_r(t)$ is independent of both $r$ and $t$.  After
deforming the metric and connections to product form near the boundary
and doubling, the local index theorem gives
\[
 H_0(t)=\int_M\Ahat(TM)\bigl(\ch(E_1)-\ch(E_0)\bigr).
\]
Finally, the long-neck assumption allows us to choose $\psi$ so that $(B_1^\psi)^2\geq\frac{\sigma}{16}$.
Hence $H_1(t)\to0$ as $t\to\infty$, and the vanishing follows from the
independence of $H_r(t)$ from $r$ and $t$.


\subsection{The boundary condition}\label{subsec:boundary}

Let $Z=Z_1\sqcup Z_0$, $\cW=\cS_1\sqcup\cS_0$,
$\pi=\pi_1\sqcup\pi_0$, and $T_\gamma=T^1_\gamma\sqcup T^0_\gamma$.
The bundle $\cW$ splits as
\[
 \cW^+=\cS_1^+\sqcup\cS_0^-,\qquad
 \cW^-=\cS_1^-\sqcup\cS_0^+.
\]
The chirality operator $\varepsilon$ equals $\varepsilon_1$ on
$Z_1$ and $-\varepsilon_0$ on $Z_0$.
An operator $R$ is called even if $\varepsilon R=R\varepsilon$ and
odd if $\varepsilon R=-R\varepsilon$.  In particular,
$D=D_1\oplus D_0$ is odd. 

Let $\nu_M$ be the inward unit normal of $\partial M$ and $\nu_a$
its lift to $\partial Z_a$.  Then $d\Phi(\nu_0)=\nu_1$.
Define a unitary boundary map by
\begin{equation}\label{eq:chi-definition}
 \chi=c_0(\nu_0)U^{-1}\varepsilon_1:
 \cS_1|_{\partial Z_1}\longrightarrow\cS_0|_{\partial Z_0}.
\end{equation}
Clifford multiplication gives
\begin{equation}\label{eq:chi-boundary-identities}
 \chi^*\chi=1,\qquad
 \chi^*c_0(\nu_0)\chi=-c_1(\nu_1),\qquad
 \chi\varepsilon_1=-\varepsilon_0\chi.
\end{equation}
We impose the boundary condition
\begin{equation}\label{eq:matching-condition}
 u_0=\chi u_1.
\end{equation}

With $\mathrm{II}(X,Y)=-g(\nabla_X\nu_M,Y)$ and
$H=\operatorname{tr}_{T\partial M}\mathrm{II}$, we define the adapted boundary operator to be
\[
 A_a=D_a^\partial
 =-\left(c_a(\nu_a)D_a+\nabla^a_{\nu_a}-\frac H2\right).
\]
Let $A=A_1\oplus A_0$ and
$c(\nu)=c_1(\nu_1)\oplus c_0(\nu_0)$.
For a tangential orthonormal frame $e_1,\ldots,e_{n-1}$ on the boundary,
\[
 \left\{-c_a(\nu_a)\sum_{j=1}^{n-1}c_a(e_j)\nabla^a_{e_j},
 c_a(\nu_a)\right\}=-Hc_a(\nu_a),
 \qquad A_ac_a(\nu_a)=-c_a(\nu_a)A_a.
\]
The covers $Z_a$ are complete as metric spaces and have a uniform collar.
The metrics on them are given by pullbacks from the compact base, and the maps $T^a_\gamma$
transport the connection data between translated charts.  Thus the
geometric hypotheses of \cite[Proposition~9.2]{BaerBandara} hold,
and each $A_a$ has a self-adjoint closure, denoted again by $A_a$.

\begin{lemma}\label{lem:matching-boundary}
We have $A_0\chi=-\chi A_1$.  Moreover,
\[
 \mathcal B_\chi
 =\{(u_1,\chi u_1):u_1\in\operatorname{Dom}(|A_1|^{1/2})\}
 \subset\operatorname{Dom}(|A_1|^{1/2})\oplus
 \operatorname{Dom}(|A_0|^{1/2})
\]
is $A$-elliptically regular and is its own adjoint boundary condition.
\end{lemma}

\begin{proof}
Since $U$ preserves the connection and Clifford multiplication,
$A_0U^{-1}=U^{-1}A_1$.  We also have $A_a\varepsilon_a=\varepsilon_aA_a$.  Hence
\[
 A_0\chi
 =-c_0(\nu_0)A_0U^{-1}\varepsilon_1
 =-c_0(\nu_0)U^{-1}\varepsilon_1A_1
 =-\chi A_1.
\]
This identity extends to the self-adjoint closures.  The spectral
theorem gives
\[
 |A_0|^{1/2}\chi=\chi|A_1|^{1/2},\qquad
 \chi\operatorname{Dom}(|A_1|^{1/2})
 =\operatorname{Dom}(|A_0|^{1/2}).
\]
Under $(u_1,u_0)\mapsto(u_1,\chi^{-1}u_0)$, the adapted operator
becomes $A_1\oplus(-A_1)$ and the boundary condition becomes the
matching boundary condition. Hence
\cite[Proposition~7.35]{BaerBandara} implies that it is $A$-elliptic regular.

For $u,v$ whose boundary values lie in $\mathcal B_\chi$,
Green's formula \cite[Theorem~2.4(iv)]{BaerBandara} gives
\[
 \langle Du,v\rangle-\langle u,Dv\rangle
 =-\sum_{a=0,1}\int_{\partial Z_a}
       \langle c_a(\nu_a)u_a,v_a\rangle
 =-\int_{\partial Z_1}
 \langle(c_1(\nu_1)+\chi^*c_0(\nu_0)\chi)u_1,v_1\rangle=0.
\]
Thus $\mathcal B_\chi\subset\mathcal B_\chi^*$.

Conversely, let $(v_1,v_0)\in\mathcal B_\chi^*$. 
$A$-elliptic regularity gives $c(\nu)^*(v_1,v_0)\in\operatorname{Dom}(|A|^{1/2})$.
Since $A_ac_a(\nu_a)=-c_a(\nu_a)A_a$, the spectral theorem gives
\[
 c_a(\nu_a)\operatorname{Dom}(|A_a|^{1/2})
 =\operatorname{Dom}(|A_a|^{1/2}),
\]
and hence $(v_1,v_0)\in\operatorname{Dom}(|A|^{1/2})$.

By the definition of the adjoint boundary condition, the right-hand
side of Green's formula vanishes for every
$(u_1,\chi u_1)\in\mathcal B_\chi$.  Taking
$u_1\in C_c^\infty(\partial Z_1,S_1)$ and changing variables by
$\Phi$, we obtain
\[
 0
 =-\int_{\partial Z_1}\langle c_1(\nu_1)u_1,v_1\rangle
   -\int_{\partial Z_0}\langle c_0(\nu_0)\chi u_1,v_0\rangle
 =-\int_{\partial Z_1}
   \langle c_1(\nu_1)u_1,v_1-\chi^*v_0\rangle.
\]
Since $c_1(\nu_1)$ is invertible and $u_1$ is arbitrary,
$v_1=\chi^*v_0$, or equivalently $v_0=\chi v_1$.  Therefore
\begin{equation}\label{eq:adjoint-boundary-condition}
 \mathcal B_\chi^*=\mathcal B_\chi.
\end{equation}
\end{proof}

Let
\[
 \mathcal D_\chi
 =\{u\in H^1(Z,\cW):u|_{\partial Z}\in\mathcal B_\chi\}.
\]
Lemma~\ref{lem:matching-boundary},
\cite[Corollaries~7.7, 7.10, and 7.15]{BaerBandara}, and the uniform
local elliptic estimate on lifts of a finite atlas of $M$ give, for
every maximal-domain section with boundary value in $\mathcal B_\chi$,
\begin{equation}\label{eq:D-H1-estimate}
 \|u\|_{H^1}\leq C_0(\|Du\|_{L^2}+\|u\|_{L^2}),
\end{equation}
and
\begin{equation}\label{eq:domain-D-chi}
 \operatorname{Dom}D_\chi
 =\{u\in\operatorname{Dom}D_{\max}:
          u|_{\partial Z}\in\mathcal B_\chi\}
 =\mathcal D_\chi.
\end{equation}
By \cite[Proposition~7.4]{BaerBandara} and
\eqref{eq:adjoint-boundary-condition}, $D_\chi$ is self-adjoint.
We write it simply as $D$.  Green's formula also shows that every
$H^1$ section satisfying \eqref{eq:matching-condition} lies in
$\operatorname{Dom}D_\chi^*=\mathcal D_\chi$.

Let $\psi\in C^\infty(M;\mathbb R)$ vanish on a neighborhood of $C$
and be constant near $\partial M$.  Choose
$C\Subset N\Subset M^\circ$ with $\psi|_N=0$, and use $\psi$ also to denote
 its pullbacks.  Define
\[
 \Sigma^\psi(u_1,u_0)
 =\bigl(\psi U\varepsilon_0u_0,
         \psi U^{-1}\varepsilon_1u_1\bigr)
\]
over $\pi^{-1}(M\setminus C)$, and set it equal to zero over
$\pi^{-1}(N)$.  These definitions agree on the overlap.  Since $U$
is parallel and unitary and preserves chirality,
\[
 \begin{gathered}
 (\Sigma^\psi)^*=\Sigma^\psi,\qquad
 (\Sigma^\psi)^2=\psi^2\operatorname{Id},\qquad
 \varepsilon\Sigma^\psi=-\Sigma^\psi\varepsilon,\qquad
 T_\gamma\Sigma^\psi=\Sigma^\psi T_\gamma,
 \end{gathered}
\]
and
\begin{equation}\label{eq:Sigma-psi-Sobolev}
 \|\Sigma^\psi u\|_{H^k}\leq C_k\|u\|_{H^k},\qquad k\geq0.
\end{equation}

Lastly, for $0\leq r\leq1$, we define a family of Callias operators
\begin{equation}\label{eq:B-family}
 B_r^\psi=D+r\Sigma^\psi,\qquad
 \operatorname{Dom}B_r^\psi=\mathcal D_\chi.
\end{equation}
We write $B_r=B_r^\psi$ when $\psi$ is fixed and understood.

\begin{lemma}\label{lem:B-properties}
There is a constant $C$, independent of $r\in[0,1]$, such that:
\begin{enumerate}[label=\textup{(\roman*)},leftmargin=2.8em]
\item\label{B:H1}
For $u\in\mathcal D_\chi$,
\begin{equation}\label{eq:D-graph-estimate}
 \|u\|_{H^1}\leq C(\|B_ru\|_{L^2}+\|u\|_{L^2}).
\end{equation}
\item\label{B:selfadjoint}
$B_r$ is self-adjoint on the fixed domain $\mathcal D_\chi$.
\item\label{B:equivariance}
$T_\gamma\mathcal D_\chi=\mathcal D_\chi$ and
$B_rT_\gamma=T_\gamma B_r$ for every $\gamma\in\Gamma$.
\item\label{B:odd}
Chirality preserves $\mathcal D_\chi$ and
$B_r\varepsilon=-\varepsilon B_r$.
\end{enumerate}
\end{lemma}

\begin{proof}
Since $r\Sigma^\psi$ is bounded and self-adjoint, 
\eqref{eq:D-H1-estimate} gives the uniform estimate and
self-adjointness holds on the same domain.
Equation~\eqref{eq:exterior-compatibility} gives
$T^0_\gamma\chi=\chi T^1_\gamma$, proving \ref{B:equivariance}.
The last identity in \eqref{eq:chi-boundary-identities} shows that
chirality preserves the domain. Both $D$ and $\Sigma^\psi$ anticommute with chirality.
\end{proof}


\subsection{Heat kernels on a fundamental domain}\label{subsec:heat-supertrace}

Choose relatively compact measurable fundamental domains $F_a\subset Z_a$
whose translates partition $Z_a$ up to measure zero, and let
$F=F_1\sqcup F_0$.  Let $1_F$ denote multiplication by the characteristic function of
$F$ on $L^2(Z,\cW)$.  For a fiber endomorphism, its supertrace is defined by
$\str A=\tr(\varepsilon A)$. For a map between fibers, we write
$|A|^2=\tr(A^*A)$.  On $L^2(Z,\cW)$, $\Tr$ denotes the ordinary
trace, and $\|\cdot\|$, $\|\cdot\|_1$, and $\|\cdot\|_2$ denote
the operator, trace, and Hilbert--Schmidt norms, respectively.  For a smoothing
operator $P$ with smooth Schwarz kernel $K_P$, define
\[
 \Str_F(P)=\int_F\str K_P(x,x)\,dx
\]
whenever this integral converges absolutely. 

For $t>0$, define
\[
 E_r(t)=e^{-tB_r^2}.
\]

\begin{lemma}\label{lem:heat-kernel}
For $j=0,1$, $r\in[0,1]$, and $s>0$, the following hold.
\begin{enumerate}[label=\textup{(\roman*)},leftmargin=2.8em]
\item\label{heat:smooth}
 $B_r^jE_r(s)$ has a unique Schwarz kernel
\[
 K_{B_r^jE_r(s)}\in
 C^\infty\!\left(Z\times Z;
 \operatorname{Hom}(\cW_y,\cW_x)\right),
\]
smooth up to the boundary in both variables.
\item\label{heat:deck}
For every $\gamma\in\Gamma$,
\begin{equation}\label{eq:heat-kernel-under-gamma}
 K_{B_r^jE_r(s)}(\gamma x,\gamma y)
 =T_{\gamma,x}K_{B_r^jE_r(s)}(x,y)T_{\gamma,y}^{-1}.
\end{equation}
\item\label{heat:L2}
We have
\begin{equation}\label{eq:heat-operator-bound}
 \|B_r^jE_r(s)\|\leq c_js^{-j/2},
 \qquad c_0=1,\quad c_1=(2e)^{-1/2},
\end{equation}
and
\begin{equation}\label{eq:heat-square-integrals}
 \int_F\!\int_Z|K_{B_r^jE_r(s)}(x,y)|^2\,dy\,dx
 =\int_Z\!\int_F|K_{B_r^jE_r(s)}(x,y)|^2\,dy\,dx<\infty.
\end{equation}
The integrals are bounded uniformly for $r\in[0,1]$ and $s$ in a
compact subinterval of $(0,\infty)$.
\end{enumerate}
\end{lemma}

\begin{proof}
We use finitely many charts on $M$ to obtain higher-regularity
estimates uniform over all lifts, including the paired boundary charts.
We use lifted interior charts where $\psi=0$, and pairs of interior
charts or boundary half-charts over $M\setminus C$ related by $\Phi$.
On a pair, we take Sobolev norms over both members and choose a smooth
cutoff $\eta$ pulled back from the same cutoff on the base. Thus
$\eta\circ\Phi=\eta$.  Multiplication by $\eta$ preserves the boundary
condition, and
\[
 B_r(\eta u)=\eta B_ru+c(d\eta)u,
 \qquad
 \|\Sigma^\psi u\|_{H^\ell(U)}
 \leq C_{U,\ell}\|u\|_{H^\ell(U)}.
\]
At the boundary, Lemma~\ref{lem:matching-boundary} identifies the
boundary condition with $\{(v,v):v\in\operatorname{Dom}(|A_1|^{1/2})\}$
for $A_1\oplus(-A_1)$.  If
$a=\sigma_{A_1}(x,\xi)$, $\xi\ne0$, then
\[
 (a\oplus(-a))(v,v)=(av,-av),\qquad
 (a\oplus(-a))(v,-v)=(av,av).
\]
Thus the principal symbol exchanges the two subspaces
$\{(v,v)\}$ and $\{(v,-v)\}$.
The compact auxiliary construction in the proof of
\cite[Theorem~7.29]{BaerBandara} therefore gives an elliptic local
boundary condition.  It is $\infty$-regular by
\cite[Corollary~2.16]{BaerBandaraGeneral}, and
\cite[Theorem~2.12]{BaerBandaraGeneral}, with nested cutoffs and the
closed graph theorem, gives the local estimate for $D$.
Using $Du=B_ru-r\Sigma^\psi u$ and induction on $\ell$ yields
\begin{equation}\label{eq:local-boundary-estimate}
 \|u\|_{H^{\ell+1}(K)}
 \leq C_{K,U,\ell}
 \bigl(\|B_ru\|_{H^\ell(U)}+\|u\|_{L^2(U)}\bigr),
\end{equation}
uniformly for $r\in[0,1]$, $\ell\geq0$, $K\Subset U$, and $u\in\mathcal D_\chi$
with $B_ru\in H^\ell(U)$.  On interior charts this is the usual
interior estimate.

The spectral theorem gives, for every integer $m\geq0$,
\begin{equation}\label{eq:spectral-heat-bound}
 E_r(s)L^2\subset\operatorname{Dom}(B_r^m),\qquad
 \|B_r^mE_r(s)\|
 \leq\sup_{\lambda\in\mathbb R}|\lambda|^me^{-s\lambda^2}.
\end{equation}
This proves \eqref{eq:heat-operator-bound}. Iteration of
\eqref{eq:local-boundary-estimate} on nested charts gives
\begin{equation}\label{eq:heat-regularity}
 \sup_{(r,s)\in[0,1]\times J}
 \|B_r^jE_r(s)u\|_{H^N(K)}\leq C_{K,J,j,N}\|u\|_{L^2}
\end{equation}
for every $N$ and $J\Subset(0,\infty)$.
Let $P=B_r^jE_r(s/2)$, $Q=E_r(s/2)$, and write
$P_xu=(Pu)(x)$, $Q_yu=(Qu)(y)$.  Sobolev embedding makes these
evaluation maps smooth in operator norm up to the boundary, and
the Schwarz kernel is given by
\[
 K_{B_r^jE_r(s)}(x,y)=P_xQ_y^*.
\]
The Schwarz kernel theorem
\cite[Theorem~5.2.1]{HormanderI} and continuity give uniqueness.
Commutation with $T_\gamma$, change of variables, and uniqueness give
\eqref{eq:heat-kernel-under-gamma}.

For an orthonormal basis $e_1,\ldots,e_d$ of $\cW_x$,
\[
 \int_Z|K_{B_r^jE_r(s)}(x,y)|^2\,dy
 =\sum_{\ell=1}^d\|QP_x^*e_\ell\|^2
 \leq d\|Q\|^2\|P_x\|^2.
\]
The right side is uniformly bounded on $\overline F$ by
\eqref{eq:heat-regularity}, and $F$ has finite volume.  Since
$B_r^jE_r(s)$ is self-adjoint, its kernel satisfies
$K(y,x)=K(x,y)^*$, proving \eqref{eq:heat-square-integrals}.
Applying a differential operator of fixed order with uniformly bounded
smooth coefficients on the left gives smooth kernels and the first
square-integrability bound, uniformly for $s\in J$. Boundedness follows
by summing the local estimates over a uniformly locally finite lifted
atlas.  The paired-chart formula for $\Sigma^\psi$ and
\eqref{eq:Sigma-psi-Sobolev} likewise give smooth kernels for the
products with $\Sigma^\psi$ used below.
\end{proof}

\begin{remark}
Since $\Sigma^\psi$ is a smooth bundle map, the products of
$B_r^jE_r(s)$ with $\Sigma^\psi$ used later in this section also
have smooth Schwarz kernels.
\end{remark}

\begin{lemma}\label{lem:trace-identities}
The following assertions hold.
\begin{enumerate}[label=\textup{(\roman*)},leftmargin=2.8em]
\item\label{trace:factors}
Let $P,L$ be bounded operators commuting with every $T_\gamma$.
Suppose $P,PL,LP$ have smooth Schwarz kernels and
$\int_F\int_Z|K_P(x,y)|^2\,dy\,dx<\infty$.  Then
\begin{equation}\label{eq:fundamental-domain-square-integrals}
 \int_F\!\int_Z|K_P(x,y)|^2\,dy\,dx
 =\int_Z\!\int_F|K_P(x,y)|^2\,dy\,dx,
\end{equation}
and
\begin{equation}\label{eq:bounded-factor}
 \max\{\|1_FPL\|_2,\|1_FLP\|_2\}
 \leq\|L\|\,\|1_FP\|_2.
\end{equation}
\item\label{trace:exchange}
Let $P,Q$ be bounded operators commuting with every $T_\gamma$,
with $P,Q,PQ,QP$ having smooth Schwarz kernels.  Suppose
$\int_F\int_Z(|K_P|^2+|K_Q|^2)<\infty$ and
\[
 \varepsilon P=(-1)^pP\varepsilon,\qquad
 \varepsilon Q=(-1)^qQ\varepsilon,\qquad p,q\in\{0,1\}.
\]
Then  $1_F\varepsilon PQ1_F$ and
$1_F\varepsilon QP1_F$ are trace class, the corresponding diagonal
supertrace integrals converge absolutely, and
\begin{equation}\label{eq:exchange}
 \Tr(1_F\varepsilon PQ1_F)=\Str_F(PQ)=(-1)^{pq}\Str_F(QP).
\end{equation}
\end{enumerate}
\end{lemma}

\begin{proof}
For \textup{(i)}, kernel uniqueness gives
$K_P(\gamma x,\gamma y)=T_{\gamma,x}K_P(x,y)T_{\gamma,y}^{-1}$.
Then we use Tonelli's theorem and  $Z=\coprod_\gamma\gamma F$ to obtain
\[
\begin{aligned}
 \int_F\!\int_Z|K_P(x,y)|^2\,dy\,dx
 &=\sum_\gamma\int_F\!\int_{\gamma F}|K_P(x,y)|^2\,dy\,dx\\
 &=\sum_\gamma\int_{\gamma^{-1}F}\!\int_F|K_P(x,y)|^2\,dy\,dx
 =\int_Z\!\int_F|K_P(x,y)|^2\,dy\,dx.
\end{aligned}
\]
Here only the unitarity of each $T_\gamma$ is used.
Thus $\|1_FP\|_2=\|P1_F\|_2$. Since the product of a Hilbert--Schmidt operator and a bounded
operator is Hilbert--Schmidt, $\|1_FPL\|_2\leq \|1_FP\|_2\|L\|$.
Applying \eqref{eq:fundamental-domain-square-integrals} to $LP$ gives
\[
 \|1_FLP\|_2=\|LP1_F\|_2
 \leq \|L\|\,\|P1_F\|_2
 =\|L\|\,\|1_FP\|_2.
\]
This proves \eqref{eq:bounded-factor}.

For \textup{(ii)}, by \textup{(i)} the operators
$1_F\varepsilon P$ and $Q1_F$ are Hilbert--Schmidt.  Hence
$1_F\varepsilon PQ1_F$ is trace class.  Moreover, Cauchy--Schwarz gives
\begin{equation}\label{eq:product-integrable}
 \int_F\!\int_Z
 |\str(K_P(x,y)K_Q(y,x))|\,dy\,dx
 \leq\|1_FP\|_2\|Q1_F\|_2<\infty.
\end{equation}
Hence,
\[
 \Tr(1_F\varepsilon PQ1_F)
 =\int_F\!\int_Z\str(K_P(x,y)K_Q(y,x))\,dy\,dx.
\]
We next identify this trace with $\Str_F(PQ)$. We choose a relatively
compact neighborhood $V$ of $\overline F$ and
$\eta_j\in C_c^\infty(V\cap Z^\circ)$ with
$0\leq\eta_j\leq1$ and $\eta_j\to1_F$ almost everywhere.
Since $V$ meets only finitely many translates of $F$, dominated
convergence gives $\eta_j\varepsilon P\to1_F\varepsilon P$ and
$Q\eta_j\to Q1_F$ in Hilbert--Schmidt norm.  Their products converge
in trace norm. Applying \cite[Proposition~2.44]{BGV} to the smooth,
compactly supported Schwarz kernel of $\eta_j\varepsilon PQ\eta_j$ yields
\[
 \Tr(1_F\varepsilon PQ1_F)
 =\lim_j\int_Z\eta_j^2\str K_{PQ}(x,x)\,dx
 =\Str_F(PQ).
\]
The same argument also applies to $QP$. Lastly, the change of variables in \textup{(i)}, applied to the absolutely
convergent integral in \eqref{eq:product-integrable}, gives
\[
\begin{aligned}
 \Str_F(PQ)
 &=\int_Z\!\int_F\str(K_P(x,y)K_Q(y,x))\,dy\,dx\\
 &=(-1)^{pq}\int_F\!\int_Z
       \str(K_Q(y,x)K_P(x,y))\,dx\,dy
 =(-1)^{pq}\Str_F(QP).
\end{aligned}
\]
The second equality uses the identity $\str(AB)=(-1)^{pq}\str(BA)$. 
\end{proof}

Applying these lemmas to $P=Q=E_r(t/2)$ defines

\begin{equation}\label{eq:H-definition}
 H_r(t)=\Str_F(E_r(t))=\Tr(1_F\varepsilon E_r(t)1_F).
\end{equation}

\vspace{0.5cm}

\begin{proposition}\label{prop:heat-invariance}
$H_r(t)$ is independent of $r\in[0,1]$ and $t>0$.
\end{proposition}

\begin{proof}
We prove $r$-independence by establishing a quadratic difference
estimate. Fix $t>0$ and $r,r+h\in[0,1]$.  For $f,g\in L^2$, consider
$u(s)=E_{r+h}(s)f$ and $v(s)=E_r(t-s)g$.  On $0<s<t$,
self-adjointness on the common domain gives
\[
 \frac d{ds}\langle u,v\rangle
 =-h\langle\Sigma^\psi u,B_rv\rangle
  -h\langle B_{r+h}u,\Sigma^\psi v\rangle.
\]
The right side above is bounded by
$C|h|\|f\|\|g\|(s^{-1/2}+(t-s)^{-1/2})$ by
\eqref{eq:heat-operator-bound}.  
Integrating over $[\delta, t-\delta]$ and letting $\delta\to 0$ gives
\begin{equation}\label{eq:paired-heat-difference}
\begin{aligned}
 E_{r+h}(t)-E_r(t)=-h\int_0^t\bigl(&B_rE_r(t-s)\Sigma^\psi E_{r+h}(s)
 +E_r(t-s)\Sigma^\psi B_{r+h}E_{r+h}(s)\bigr)\,ds.
\end{aligned}
\end{equation}
We next estimate the trace norms of the two operators in the
integrand.  If $0<s\leq t/2$, write $E_r(t-s)=E_r((t-s)/2)E_r((t-s)/2)$,
while if $t/2\leq s<t$, write $E_{r+h}(s)=E_{r+h}(s/2)E_{r+h}(s/2)$.
In either case, the parameter of each of the two equal factors is at
least $t/4$.  Lemmas~\ref{lem:heat-kernel} and
\ref{lem:trace-identities}\ref{trace:factors}, together with
$\|AC\|_1\leq\|A\|_2\|C\|_2$, therefore give
\begin{equation}\label{eq:heat-variation-bound}
\begin{aligned}
 &\|1_F\varepsilon B_rE_r(t-s)\Sigma^\psi E_{r+h}(s)1_F\|_1
+
 \|1_F\varepsilon E_r(t-s)\Sigma^\psi
 B_{r+h}E_{r+h}(s)1_F\|_1\\
 &\leq
 C_t\bigl(1+s^{-1/2}+(t-s)^{-1/2}\bigr).
\end{aligned}
\end{equation}
The bound in \eqref{eq:heat-variation-bound} is integrable on
$(0,t)$, so the integral converges in trace norm.  We may therefore
take the trace under the integral.  Lemma
\ref{lem:trace-identities}\ref{trace:exchange} gives
\[
\begin{aligned}
 \Str_F(B_rE_r(t-s)\Sigma^\psi E_{r+h}(s))
 &=-\Str_F(\Sigma^\psi E_{r+h}(s)B_rE_r(t-s)),\\
 \Str_F(E_r(t-s)\Sigma^\psi B_{r+h}E_{r+h}(s))
 &=\Str_F(\Sigma^\psi E_{r+h}(s)B_{r+h}E_r(t-s)).
\end{aligned}
\]
Subtracting $B_r$ from $B_{r+h}$ gives
\begin{equation}\label{eq:heat-quadratic-difference}
 H_{r+h}(t)-H_r(t)
 =-h^2\int_0^t\Str_F\bigl(
 \Sigma^\psi E_{r+h}(s)\Sigma^\psi E_r(t-s)\bigr)\,ds.
\end{equation}
For $0<s\leq t/2$, write
\[
 \Sigma^\psi E_{r+h}(s)\Sigma^\psi E_r(t-s)
 =
 \bigl[\Sigma^\psi E_{r+h}(s)\Sigma^\psi
 E_r((t-s)/2)\bigr]E_r((t-s)/2),
\]
while for $t/2\leq s<t$, write
\[
 \Sigma^\psi E_{r+h}(s)\Sigma^\psi E_r(t-s)
 =
 \bigl[\Sigma^\psi E_{r+h}(s/2)\bigr]
 \bigl[E_{r+h}(s/2)\Sigma^\psi E_r(t-s)\bigr].
\]
In the first case $(t-s)/2\geq t/4$, and in the second
$s/2\geq t/4$.  Lemma~\ref{lem:heat-kernel}\textup{(iii)},
\eqref{eq:bounded-factor}, and \eqref{eq:product-integrable} therefore give
\[
 \left|
 \Str_F\bigl(
 \Sigma^\psi E_{r+h}(s)\Sigma^\psi E_r(t-s)
 \bigr)
 \right|
 \leq C_t
\]
uniformly for $r,r+h\in[0,1]$ and $0<s<t$.  Hence
\[
 |H_{r+h}(t)-H_r(t)|\leq C_t h^2.
\]
Subdivision of
$[r_0,r_1]$ into $N$ equal intervals gives
$|H_{r_1}(t)-H_{r_0}(t)|\leq C_t(r_1-r_0)^2/N$. Taking $N\to\infty$ proves
independence of $r$.

To show independence in $t$, we consider
\[
 R_h=\frac{E_r(t/2+h)-E_r(t/2)}h+B_r^2E_r(t/2).
\]
For $0<|h|<t/4$, the spectral theorem gives
$\|R_h\|\leq\frac{|h|}{2}\sup_\lambda\lambda^4e^{-t\lambda^2/4}\to0$.
By the semigroup property,
\[
 \frac{E_r(t+h)-E_r(t)}h+B_r^2E_r(t)
 =E_r(t/4)R_hE_r(t/4).
\]
By Lemma~\ref{lem:heat-kernel}, $E_r(t/4)1_F$ is
Hilbert--Schmidt, and hence
\[
 \bigl\|
 1_F\varepsilon E_r(t/4)R_hE_r(t/4)1_F
 \bigr\|_1
 \leq
 \|E_r(t/4)1_F\|_2^2\,\|R_h\|
 \longrightarrow0.
\]
Thus
\[
 H_r'(t)=-\Str_F(B_r^2E_r(t)).
\]
Set $P=B_rE_r(t/2)$.  Since $P$ anticommutes with chirality and
$P^2=B_r^2E_r(t)$, Lemma~\ref{lem:trace-identities}\ref{trace:exchange}
gives $\Str_F(P^2)=-\Str_F(P^2)=0$.  Therefore $H_r'(t)=0$.
\end{proof}


\subsection{Identification with the characteristic number}

\begin{lemma}\label{lem:collar-deformation}
There are $\delta>0$, a metric $g_1$ on $M$, and unitary connections
$\nabla^{E_a,1}$ with the following properties.
\begin{enumerate}[label=\textup{(\roman*)},leftmargin=2.8em]
\item\label{collar:product}
The collar $[0,\delta]\times\partial M$ lies in $M\setminus C$.  Using
parallel transport for the original connections to identify the
bundles along the collar,
\[
 g_1=dx^2+g|_{T\partial M},\qquad
 \nabla^{E_a,1}=\partial_x+\nabla^{E_a}|_{\partial Z_a}.
\]
\item\label{collar:fixed}
The metric and connections are unchanged on a neighborhood of $C$
and its lifts.
\item\label{collar:exterior}
On $\pi_0^{-1}(M\setminus C)$,
$U^E\nabla^{E_0,1}=\nabla^{E_1,1}U^E$.
The forms $\ch(E_a,\nabla^{E_a,1})$ descend to $M$ and agree on
$M\setminus C$.
\item\label{collar:domain}
On $\cW_1=(S_{Z_1,g_1}\otimes E_1)\sqcup
(S_{Z_0,g_1}\otimes E_0)$, with opposite chirality over $Z_0$,
$D^{(1)}=D_{1,1}\oplus D_{0,1}$ is self-adjoint on
\[
 \{u\in H^1(Z,\cW_1):u|_{\partial Z}\in\mathcal B_\chi\}.
\]
\item\label{collar:heat}
For every $t>0$,
\begin{equation}\label{eq:collar-product-heat}
 \Str_F(e^{-t(D^{(1)})^2})=H_0(t).
\end{equation}
\end{enumerate}
\end{lemma}

\begin{proof}
Choose $\delta>0$ with $[0,4\delta)\times\partial M\subset M\setminus C$.
Using parallel transport along the normal direction, we can write $g=dx^2+g_x$ and
$\nabla^{E_a}=\partial_x+\nabla_x^{E_a}$ in the collar.  
Choose a smooth cutoff $\vartheta$ with values in $[0,1]$, equal to $1$ on
$[0,\delta]$ and $0$ on $[2\delta,4\delta)$, and set
\begin{equation}\label{eq:collar-deformation}
 h_\rho(x)=(1-\rho\vartheta(x))x,\quad
 g_\rho=dx^2+g_{h_\rho(x)},\quad
 \nabla^{E_a,\rho}=\partial_x+\nabla^{E_a}_{h_\rho(x)}.
\end{equation}
Outside the collar, we extend $(g_\rho, \nabla^{E_a,\rho})$ by $(g, \nabla^{E_a})$.  This gives
\ref{collar:product} and \ref{collar:fixed}.  Since $U^E$ is independent of $x$ and satisfies
$U^E\nabla_x^{E_0}=\nabla_x^{E_1}U^E$ for every $x$, it remains parallel for the deformed connections on
$M\setminus C$.

Let $\beta_\rho^a:S_{Z_a,g}\to S_{Z_a,g_\rho}$ be the
Bourguignon--Gauduchon identification, chosen continuously from the
identity; see \cite[\S I, Propositions~1 and 5, and \S II]{BourguignonGauduchon}.
It is the identity at $x=0$, where the metric, normal vector, and
twisting connection are unchanged.  Thus the transported map
$U_\rho=(\beta_\rho^1\otimes1)U(\beta_\rho^0\otimes1)^{-1}$
defines the same boundary map $\chi$. The tangential spin connection changes because the normal derivative
of the metric changes.  Its contribution to $A_{a,\rho}$ is canceled
by the change of $H_\rho/2$.
Moreover, $h_\rho'(0)=1-\rho$, so
$\mathrm{II}_\rho=(1-\rho)\mathrm{II}$ and $H_\rho=(1-\rho)H$.
For a tangential orthonormal frame, the spinorial Gauss formula gives
\[
 (\beta_\rho^a\otimes1)^{-1}\nabla^{a,\rho}_{e_j}
       (\beta_\rho^a\otimes1)
 =\nabla^a_{e_j}+\frac\rho2\sum_k
       \mathrm{II}(e_j,e_k)c_a(\nu_a)c_a(e_k).
\]
Consequently the adapted boundary operator satisfies
\[
 (\beta_\rho^a\otimes1)^{-1}A_{a,\rho}(\beta_\rho^a\otimes1)-A_a
 =\frac{\rho H}{2}+\frac{H_\rho-H}{2}=0.
\]
Write $\cW_\rho=(S_{Z_1,g_\rho}\otimes E_1)\sqcup
(S_{Z_0,g_\rho}\otimes E_0)$, with opposite chirality on $Z_0$.
Lemma~\ref{lem:matching-boundary} therefore applies to $D^{(\rho)}$.
The metrics $g_\rho$ are complete, and for $\rho\in[0,1]$, the
coefficients of $D^{(\rho)}$ and their derivatives are uniformly
bounded in a finite lifted atlas. Moreover, $\|\sigma_{D^{(\rho)}}(z,\xi)\|=|\xi|_{g_\rho}$.
Hence the argument leading to \eqref{eq:D-H1-estimate} and
\eqref{eq:domain-D-chi} applies uniformly in $\rho$, giving
\begin{equation}\label{eq:rho-H1-estimate}
 \|u\|_{H^1(g_\rho)}
 \leq C\bigl(\|D^{(\rho)}u\|_{L^2(g_\rho)}
             +\|u\|_{L^2(g_\rho)}\bigr).
\end{equation}
and
\[
 \operatorname{Dom}D^{(\rho)}
 =\mathcal D_\chi^{(\rho)}
 =\{u\in H^1(Z,\cW_\rho):
      u|_{\partial Z}\in\mathcal B_\chi\}.
\]
The self-adjointness argument after \eqref{eq:domain-D-chi}
applies to this domain.

The maps $T_\gamma^{a,\rho}=(\beta_\rho^a\otimes1)
T_\gamma^a(\beta_\rho^a\otimes1)^{-1}$ preserve the deformed
Clifford multiplication, chirality, and connection.  Indeed, locally
$T_\gamma^a=\widehat T_\gamma^a\otimes t_\gamma^a$, where
$\widehat T_\gamma^a$ is a spin lift and $t_\gamma^a$ preserves
$\nabla^{E_a}$, by \eqref{eq:T-properties} and irreducibility of the
complex Clifford representation.  Using parallel transport along the normal direction, $t_\gamma^a$
is independent of $x$ and satisfies $t_\gamma^a\nabla_x^{E_a} =\nabla_x^{E_a}t_\gamma^a$
for every $x$, and hence also preserves the deformed connection
$\partial_x+\nabla_{h_\rho(x)}^{E_a}$.  Since both $g$ and $g_\rho$
are preserved by $\gamma$, the fiberwise isometry
$(TZ_a,g)\to(TZ_a,g_\rho)$ defining $\beta_\rho^a$ commutes with
$d\gamma$.  Therefore $\beta_\rho^a\widehat T_\gamma^a(\beta_\rho^a)^{-1}$
is a spin lift of the isometry
$\gamma:(Z_a,g_\rho)\to(Z_a,g_\rho)$ and preserves the corresponding
spin connection.
Thus \eqref{eq:twisting-curvature-under-gamma} applies to the deformed
connections, proving the descent in \ref{collar:exterior}. Their
agreement outside $C$ follows from the parallelism of $U^E$.

We next define $\beta_\rho=\beta_\rho^1\sqcup\beta_\rho^0$,
$\mu_\rho=d\operatorname{vol}_{g_\rho}/d\operatorname{vol}_g$, and
\[
 J_\rho u=\mu_\rho^{1/2}(\beta_\rho^{-1}\otimes1)u,\qquad
 \widetilde D_\rho=J_\rho D^{(\rho)}J_\rho^{-1},\qquad
 S_\rho(s)=e^{-s\widetilde D_\rho^2}.
\]
The maps $J_\rho$ are unitary on $L^2$ and uniformly bounded
isomorphisms on $H^1$.  Since $J_\rho=1$ at the boundary,
$\widetilde D_\rho$ is self-adjoint on the common domain
$\mathcal D_\chi$ and anticommutes with chirality.  Moreover,
$D^{(\rho)}T_\gamma^{a,\rho}
=T_\gamma^{a,\rho}D^{(\rho)}$.  Since both $g$ and $g_\rho$ are
$\Gamma$-invariant, so is $\mu_\rho$, and the definitions give
\[
 J_\rho T_\gamma^{a,\rho}J_\rho^{-1}=T_\gamma^a.
\]
Consequently $\widetilde D_\rho T_\gamma=T_\gamma\widetilde D_\rho$ for every $\gamma\in\Gamma$.
The coefficients of $\widetilde D_\rho$ depend smoothly on $\rho$,
with uniform bounds on the lifts of a fixed finite atlas of $M$.
Using \eqref{eq:rho-H1-estimate}, the proof of
\eqref{eq:local-boundary-estimate} applies uniformly in $\rho$.
Iterating as in \eqref{eq:heat-regularity} gives, for
$K\Subset U$ in any lifted interior or boundary chart,
$j=0,1$, $N\geq1$, and $J\Subset(0,\infty)$,
\begin{equation}\label{eq:rho-heat-regularity}
 \sup_{\rho\in[0,1],\,s\in J}
 \|\widetilde D_\rho^jS_\rho(s)u\|_{H^N(K)}
 \leq C_{K,J,j,N}\|u\|_{L^2}.
\end{equation}
The evaluation-map argument in Lemma~\ref{lem:heat-kernel} therefore
gives smooth Schwarz kernels and, uniformly for
$\rho\in[0,1]$, $s\in I\Subset(0,\infty)$, and $j=0,1$,
\[
 \|1_F\widetilde D_\rho^jS_\rho(s)\|_2
 +\|\widetilde D_\rho^jS_\rho(s)1_F\|_2
 \leq C_I.
\]
Lemma~\ref{lem:trace-identities}, applied to
$S_\rho(t/2)^2$, now shows that
\[
 \mathscr H_\rho(t)=\Str_F(S_\rho(t))
\]
is well-defined and absolutely convergent.

We next prove that $\mathscr H_\rho(t)$ is independent of $\rho$.
For $\rho,\rho+h\in[0,1]$, set $L_{\rho,h}=\widetilde D_{\rho+h}-\widetilde D_\rho$.
Smooth dependence of the coefficients gives
\[
 \|L_{\rho,h}u\|_{L^2}
 \leq C|h|\|u\|_{H^1}.
\]
Since $J_\tau^{\pm1}$ are uniformly bounded on $H^1$,
\eqref{eq:rho-H1-estimate} and the spectral theorem give, for
$\tau=\rho,\rho+h$,
\[
 \|S_\tau(s)\|\leq1,\qquad
 \|\widetilde D_\tau S_\tau(s)\|
 \leq(2es)^{-1/2},
\]
and hence
\[
 \|L_{\rho,h}S_\tau(s)f\|
 \leq C|h|(1+s^{-1/2})\|f\|.
\]
Using \eqref{eq:rho-heat-regularity} with one additional derivative
and the argument proving \eqref{eq:heat-square-integrals}, we obtain
\[
 \|1_FL_{\rho,h}S_\tau(s)\|_2
 +\|L_{\rho,h}S_\tau(s)1_F\|_2
 \leq C_I|h|,
 \qquad s\in I\Subset(0,\infty),
\]
for $\tau=\rho,\rho+h$.  The operator $L_{\rho,h}$ is symmetric on
$\mathcal D_\chi$, anticommutes with chirality, and commutes with every
$T_\gamma$. We interpret $S_\tau(s)L_{\rho,h}$ as
$(L_{\rho,h}S_\tau(s))^*$.

We now repeat the proof of the independence of $H_r(t)$ from $r$ in
Proposition~\ref{prop:heat-invariance}, with $L_{\rho,h}$ in place of
$h\Sigma^\psi$.  The preceding estimates justify taking the trace
under the integral, and
Lemma~\ref{lem:trace-identities}\ref{trace:exchange}, together with
$\widetilde D_{\rho+h}-\widetilde D_\rho=L_{\rho,h}$, gives
\[
 \mathscr H_{\rho+h}(t)-\mathscr H_\rho(t)
 =-\int_0^t\Str_F\bigl(
 L_{\rho,h}S_{\rho+h}(s)L_{\rho,h}S_\rho(t-s)\bigr)\,ds.
\]
Proceed similarly as in the proof of Proposition~\ref{prop:heat-invariance}, the preceding bounds give, for $0<s<t$,
\[
 \left|\Str_F\bigl(
 L_{\rho,h}S_{\rho+h}(s)L_{\rho,h}S_\rho(t-s)\bigr)\right|
 \leq C_t h^2\bigl(1+s^{-1/2}+(t-s)^{-1/2}\bigr).
\]
Integrating this bound gives
\[
 |\mathscr H_{\rho+h}(t)-\mathscr H_\rho(t)|
 \leq C_t h^2.
\]
Subdivision shows that $\mathscr H_\rho(t)$ is independent of $\rho$.
Finally, conjugation by $J_\rho$ preserves the supertrace density:
\[
 \str K_{S_\rho(t)}^g(z,z)\,d\operatorname{vol}_g
 =\str K_{e^{-t(D^{(\rho)})^2}}^{g_\rho}(z,z)
       \,d\operatorname{vol}_{g_\rho}.
\]
At $\rho=0,1$ this proves \eqref{eq:collar-product-heat}.
\end{proof}

\begin{lemma}\label{lem:doubled-local-formula}
For the data of Lemma~\ref{lem:collar-deformation}, we have
\begin{equation}\label{eq:doubled-local-formula}
 \begin{aligned}
 \lim_{\tau\downarrow0}\Str_F(e^{-\tau(D^{(1)})^2})
=\int_M\Ahat(TM,g_1)
  \bigl(\ch(E_1,\nabla^{E_1,1})-\ch(E_0,\nabla^{E_0,1})\bigr).
 \end{aligned}
\end{equation}
\end{lemma}

\begin{proof}
We consider the doubles
\[
 \widehat M=M\cup_{\partial M}(-M),\qquad
 \widehat Z=Z_1\cup_\Phi(-Z_0),\qquad
 \widehat E=E_1\cup_{(U^E)^{-1}}E_0.
\]
The product structures give a smooth metric and twisting connection
on the double.  Since $\Phi$ is $\Gamma$-equivariant, the maps
$\pi_1$ and $\pi_0$ glue to a principal $\Gamma$-bundle $\widehat\pi:\widehat Z\longrightarrow\widehat M$.
The image $\widehat F$ of $F_1\sqcup F_0$ under the natural
inclusions into $\widehat Z$ is a relatively compact measurable
fundamental domain.  Under the gluing of $Z_1$ to $-Z_0$, a tangential vector
$w\in T\partial Z_0$ is identified with $\Phi_*w\in T\partial Z_1$,
while the inward normal $\nu_1$ is identified with $-\nu_0$.
The definition of $\chi$ then gives
\begin{equation}\label{eq:double-identities}
 \begin{gathered}
 \chi c_1(\Phi_*w)=c_0(w)\chi,\qquad
 \chi c_1(\nu_1)=-c_0(\nu_0)\chi,\\
 \chi\varepsilon_1=-\varepsilon_0\chi,\qquad
 T_\gamma^{0,1}\chi=\chi T_\gamma^{1,1}.
 \end{gathered}
\end{equation}
The boundary map $\chi$ is
parallel in the tangential direction and extends constantly in the normal
direction.  It glues the twisted spinor bundles to
\[
 \widehat\cS=(S_{Z_1,g_1}\otimes E_1)
       \cup_\chi(S_{Z_0,g_1}\otimes E_0).
\]
The first three identities glue Clifford multiplication and chirality,
and parallelism glues the connection.  Let $\widehat D$ be the resulting
Dirac operator, with domain $H^1(\widehat Z,\widehat\cS)$.
Locally choose a spin lift of $\Phi$, and let $\Phi_{\mathrm{spin}}$
be the induced map on spinors.  Since both $\widehat\Phi$ and
$\Phi_{\mathrm{spin}}$ preserve Clifford multiplication and the spin
connection, we have $\widehat\Phi=\lambda\,\Phi_{\mathrm{spin}}$
for a constant $\lambda\in\mathbb C$ with $|\lambda|=1$.  These constant
factors do not contribute to the curvature, so the twisting curvature
in the local index formula is that of $\widehat E$.

Let $\mathcal R:L^2(\widehat Z,\widehat{\cS})\longrightarrow L^2(Z,\cW_1;g_1)$
be the map which restricts a section to the two halves of
$\widehat Z$.  Since the common boundary has measure zero,
$\mathcal R$ is unitary, and it preserves chirality.

A pair $u=(u_1,u_0)$ with $u_a\in H^1(Z_a)$ determines an
$H^1$ section of $\widehat{\cS}$ precisely when $u_0=\chi u_1$ on $\partial Z$.
Hence $\mathcal R\operatorname{Dom}(\widehat D)=\operatorname{Dom}(D^{(1)})$.
On the interiors of the two halves, $\widehat D$ agrees with
$D_{1,1}$ and $D_{0,1}$, respectively.  Therefore
\[
 \mathcal R\widehat D\mathcal R^{-1}=D^{(1)}.
\]
Consequently,
\[
 \Str_F(e^{-\tau(D^{(1)})^2})
 =\int_{\widehat F}
   \str K_{e^{-\tau\widehat D^2}}(z,z)\,dz.
\]
The last identity in \eqref{eq:double-identities} glues the maps
$T_\gamma^{a,1}$ to maps $\widehat T_\gamma$ preserving the Clifford
multiplication, chirality, and connection.  

The complete metric on $\widehat Z$ is pulled back from the compact
manifold $\widehat M$, and hence has positive injectivity radius and
uniformly bounded curvature and its derivatives.  Since
$\widehat T_\gamma$ preserves the connection on $\widehat{\cS}$,
the curvature of $\widehat{\cS}$ and all its covariant derivatives
are uniformly bounded on $\widehat Z$.
Thus \cite[Proposition~2.11]{RoeOpenIndex} gives a uniform local
heat-kernel expansion.  The local coefficient calculation of
\cite[Theorems~4.1--4.2 and \S4.3]{BGV}, with the normalizations in
the Introduction, gives
\[
 \begin{gathered}
 \str K_{e^{-\tau\widehat D^2}}(z,z)\,d\operatorname{vol}_{\widehat g}
 \sim\sum_{j\geq0}\tau^{(j-n)/2}\str\Psi_j,\\
 \str\Psi_j=0\quad(j<n),\qquad
 \str\Psi_n=[\Ahat(T\widehat Z)\ch(\widehat E)]_{[n]}.
 \end{gathered}
\]
These coefficient identities depend only on the local data.
The remainder after the constant term is uniformly
$O(\tau^{1/2})$, so it tends to zero after integration over
$\widehat F$.  On the copy of $-Z_0$ in $\widehat Z$, the orientation
is opposite to that of $Z_0$, and the corresponding chirality on
$\cW_1$ is $-\varepsilon_0$.  Since the Chern forms descend to $M$
and each $\pi_a$ is one-to-one almost everywhere on $F_a$, we obtain
\[
 \begin{aligned}
&\int_{\widehat F}[\Ahat(T\widehat Z)\ch(\widehat E)]_{[n]}\\
 &=\int_{F_1}
 [\pi_1^*\Ahat(TM,g_1)\ch(E_1,\nabla^{E_1,1})]_{[n]}
-\int_{F_0}
 [\pi_0^*\Ahat(TM,g_1)\ch(E_0,\nabla^{E_0,1})]_{[n]},
 \end{aligned}
\]
which is the right-hand side of
\eqref{eq:doubled-local-formula}.
\end{proof}

\begin{proposition}\label{prop:boundary-formula}
For every $\psi$ used in \eqref{eq:B-family} and every $t>0$,
\begin{equation}\label{eq:boundary-formula}
 \Str_F(E_1(t))=
 \int_M\Ahat(TM)\bigl(\ch(E_1)-\ch(E_0)\bigr).
\end{equation}
\end{proposition}

\begin{proof}
Proposition~\ref{prop:heat-invariance} and
Lemma~\ref{lem:collar-deformation} give, for every $\tau>0$,
\[
 H_1(t)=H_0(\tau)=\Str_F(e^{-\tau(D^{(1)})^2}).
\]
Letting $\tau\to 0$ and applying
Lemma~\ref{lem:doubled-local-formula}, we obtain the characteristic
integral for the deformed metric and connections.  Since $\ch(E_1,\nabla^{E_1,1})-\ch(E_0,\nabla^{E_0,1})$
vanishes on $M\setminus C$, its integrand is supported in $C$.
On a neighborhood of $C$, the metric and connections agree with the
original ones.  Hence the deformed characteristic integral equals
\[
 \int_M\Ahat(TM)\bigl(\ch(E_1)-\ch(E_0)\bigr),
\]
which proves \eqref{eq:boundary-formula}.
\end{proof}


\subsection{The long-neck estimate}

\begin{proposition}\label{prop:long-neck}
Under the assumptions of Theorem~\ref{thm:vanishing}, there is a
nonnegative function $\psi\in C^\infty(M)$, zero near $C$ and constant near
$\partial M$, such that
\begin{equation}\label{eq:gap}
 \|B_1^\psi u\|_{L^2}^2\geq\frac{\sigma}{16}\|u\|_{L^2}^2,
 \qquad u\in\mathcal D_\chi.
\end{equation}
\end{proposition}

\begin{proof}
Define 
\[
 \lambda_0=\frac{\sigma}{16},\qquad
 L_0=\frac{\pi}{2\sqrt{\lambda_0}}
     =\frac{2\pi}{\sqrt\sigma},
\]
and choose $\eta\in(0,1)$ and $R$ with
$(1+\eta)L_0<R<\dist(C,\partial M)$.
Lemma~\ref{lem:smooth-plateau} gives
$r\in C^\infty(M;[0,R])$ satisfying
\[
 r=0\ \text{near }C,\qquad
 r=R\ \text{near }\partial M,\qquad |dr|\leq1+\eta.
\]
Set $m_0=1+\max_{\partial M}\max\{0,-H/2\}$.
Since $\arctan(m_0/\sqrt{\lambda_0})<\pi/2$ and
$(1+\eta)L_0<R$, we may choose
$\theta\in C^\infty([0,R];[0,1])$, which is zero near $R$, such that
\[
 \frac{1+\eta}{\sqrt{\lambda_0}}
   \arctan\frac{m_0}{\sqrt{\lambda_0}}
 <\int_0^R\theta(s)\,ds<(1+\eta)L_0.
\]
Define
\[
 \Theta(t)=\int_0^t\theta(s)\,ds,\qquad
 \varphi(t)=\sqrt{\lambda_0}\tan\left(
  \frac{\sqrt{\lambda_0}}{1+\eta}\Theta(t)\right),
 \qquad \psi=\varphi\circ r.
\]
The function $\varphi$ is
nonnegative, smooth, constant near $R$, and satisfies
\[
 \varphi(0)=0,\qquad \varphi(R)>m_0,\qquad
 \varphi'(t)=\frac{\theta(t)}{1+\eta}
              \bigl(\varphi(t)^2+\lambda_0\bigr).
\]
Consequently, $\psi$ is nonnegative and smooth, zero near $C$,
and constant near $\partial M$, with
\begin{equation}\label{eq:psi-estimates}
 |d\psi|\leq\psi^2+\frac{\sigma}{16},
 \qquad\text{and}\quad 
 \frac H2+\psi>0\quad\text{on }\partial M.
\end{equation}

For $u$ smooth up to the boundary with compact support, integrating
\eqref{eq:SL} gives
\[
 \|Du\|_{L^2}^2
 =\|\nabla u\|_{L^2}^2
  +\int_Z\left\langle
   \left(\frac{\Sc}{4}+\cR\right)u,u\right\rangle
  +\int_{\partial Z}\left(\frac H2|u|^2-\langle Au,u\rangle\right),
\]
where $\cR=\operatorname{diag}(\cR^{E_1},\cR^{E_0})$ and the
normal is inward.  Imposing the boundary condition $u_0=\chi u_1$ and using the identity
$A_0\chi=-\chi A_1$, and changing variables by $\Phi$, we obtain
\[
 \int_{\partial Z}\langle Au,u\rangle
 =\int_{\partial Z_1}\langle A_1u_1,u_1\rangle
  -\int_{\partial Z_0}\langle\chi A_1u_1,\chi u_1\rangle=0.
\]

The identities $D_1U=UD_0$ and
$D_a\varepsilon_a=-\varepsilon_aD_a$ give
\[
 (D\Sigma^\psi+\Sigma^\psi D)u
 =\bigl(c_1(d\psi)U\varepsilon_0u_0,\,
         c_0(d\psi)U^{-1}\varepsilon_1u_1\bigr)
\]
over $\pi^{-1}(M\setminus C)$, and the expression vanishes near $\pi^{-1}(C)$.  Since $\Phi$
preserves volume and the pullback of $|d\psi|$, the inequality
$2ab\leq a^2+b^2$ yields
\[
 \left|\int_Z\langle(D\Sigma^\psi+\Sigma^\psi D)u,u\rangle\right|
 \leq\int_Z|d\psi|\,|u|^2.
\]
Moreover, the definition of $\chi$ gives
\[
 \begin{aligned}
 U\varepsilon_0u_0
 &=U\varepsilon_0c_0(\nu_0)U^{-1}\varepsilon_1u_1
   =-c_1(\nu_1)u_1,\\
 c_0(\nu_0)u_0&=-U^{-1}\varepsilon_1u_1.
 \end{aligned}
\]
Thus $\Sigma^\psi u=-\psi c(\nu)u$ on $\partial Z$.
Green's formula therefore gives
\[
 \begin{aligned}
 2\operatorname{Re}\langle Du,\Sigma^\psi u\rangle_{L^2}
 &=\int_Z\langle(D\Sigma^\psi+\Sigma^\psi D)u,u\rangle
   -\int_{\partial Z}\langle c(\nu)u,\Sigma^\psi u\rangle\\
 &\geq-\int_Z|d\psi|\,|u|^2
        +\int_{\partial Z}\psi|u|^2.
 \end{aligned}
\]
Combining these identities with $(\Sigma^\psi)^2=\psi^2I$ and
$\Sc/4+\cR\geq\sigma/8$, we obtain
\begin{equation}\label{eq:energy}
 \begin{aligned}
 \|B_1^\psi u\|_{L^2}^2
 \geq{}&\|\nabla u\|_{L^2}^2
 +\int_Z\left(\frac{\sigma}{8}+\psi^2-|d\psi|\right)|u|^2
 +\int_{\partial Z}\left(\frac H2+\psi\right)|u|^2.
 \end{aligned}
\end{equation}
By density, the trace theorem, and the boundedness
$A:H^{1/2}\to H^{-1/2}$, \eqref{eq:energy} extends to every
$u\in\mathcal D_\chi$.

Using \eqref{eq:psi-estimates} in \eqref{eq:energy} gives
\[
 \|B_1^\psi u\|_{L^2}^2
 \geq\frac{\sigma}{16}\|u\|_{L^2}^2,
 \qquad u\in\mathcal D_\chi,
\]
which proves \eqref{eq:gap}.
\end{proof}


\subsection{Proof of the vanishing theorem}
\begin{proof}[Proof of Theorem~\ref{thm:vanishing}]
Choose $\psi$ from Proposition~\ref{prop:long-neck}, and set
$B_1=B_1^\psi$ and $\delta=\sigma/16$.  Then $B_1^2\geq\delta$.
For $t\geq1$, the spectral theorem gives
\[
 0\leq E_1(t)\leq e^{-(t-1)\delta}E_1(1).
\]
Multiplying on the left and right by $1_F$ gives
\[
 0\leq1_FE_1(t)1_F
 \leq e^{-(t-1)\delta}1_FE_1(1)1_F.
\]
Lemma~\ref{lem:trace-identities}, applied to
$P=Q=E_1(t/2)$, shows that
$1_F\varepsilon E_1(t)1_F$ is trace class for every $t>0$.
Since $\varepsilon$ commutes with $1_F$ and $\varepsilon^2=1$,
\[
 1_FE_1(t)1_F
 =\varepsilon\bigl(1_F\varepsilon E_1(t)1_F\bigr)
\]
is trace class as well.  Taking traces in the preceding inequality,
we obtain
\[
 \begin{aligned}
 |H_1(t)|
 \leq\Tr(1_FE_1(t)1_F)
 \leq e^{-(t-1)\delta}\Tr(1_FE_1(1)1_F)
 \longrightarrow0
 \qquad \text{as}\ t\to\infty.
 \end{aligned}
\]
Proposition~\ref{prop:heat-invariance} shows that $H_1(t)$ is
independent of $t$, and hence $H_1(t)=0$.  Proposition~\ref{prop:boundary-formula}
then gives
\[
 \int_M\Ahat(TM)\bigl(\ch(E_1)-\ch(E_0)\bigr)=0,
\]
which proves Theorem~\ref{thm:vanishing}.
\end{proof}


\vspace{0.5cm}

\section{Proof under the relative $K$-area assumption}\label{sec:K}

Assume alternative~\ref{alt:K} in Theorem~\ref{thm:main} and, for a contradiction, fix a
complete metric $g$ on $X$ with $\Sc_g\geq\sigma>0$.  We construct almost-flat bundles and apply Theorem~\ref{thm:vanishing}
to obtain contradiction.

\begin{lemma}\label{lem:K-compact}
There is a constant $b_n>0$, depending only on $n$, such that, for all
sufficiently large $i$, there are connected compact smooth domains
$M_i\subset X^\circ$, compact smooth domains
$C_i\Subset M_i^\circ$ (possibly disconnected), equal-rank Hermitian
bundles $A_{1,i},A_{0,i}\to M_i$, and a unitary map
\[
 V_i^A:A_{0,i}\longrightarrow A_{1,i}
 \qquad\text{over }M_i\setminus C_i
\]
with the following properties:
\begin{enumerate}[label=\textup{(K\arabic*)},ref=\textup{(K\arabic*)},
                  leftmargin=3.4em]
\item\label{K:distance}
$\partial M_i\ne\varnothing$ and
$\dist(C_i,\partial M_i)\to\infty$.

\item\label{K:exterior}
The bundles are flat on $M_i\setminus C_i$, and $V_i^A$ is parallel
there.

\item
$\max_a\|R^{A_{a,i}}\|_g\leq b_n/i$.

\item\label{K:pairing}
\[
 \int_{M_i}\Ahat(TM_i)
 \bigl(\ch(A_{1,i})-\ch(A_{0,i})\bigr)\ne0.
\]
\end{enumerate}
\end{lemma}

\begin{proof}
By the  formulation in
\cite[Theorem~5a]{GromovFourLectures}, applied with
$\varepsilon=1/i$, there are Hermitian bundles
$\widetilde B_{1,i},\widetilde B_{0,i}\to X^\circ$ with unitary
connections such that
\[
 \max_a\|R^{\widetilde B_{a,i}}\|_g\leq\frac1i,
 \qquad R^{\widetilde B_{a,i}}=0
 \quad\text{on }X^\circ\setminus X_i^\circ,
\]
a parallel isomorphism
$\widetilde V_i:\widetilde B_{0,i}\to\widetilde B_{1,i}$ on that
set, and a nonzero relative Chern number of the virtual bundle
\[
 \xi_i=\widetilde B_{1,i}-\widetilde B_{0,i}.
\]
Here a virtual bundle denotes a formal difference of actual bundles.
Since $\widetilde V_i^*\widetilde V_i$ is parallel, replacing
$\widetilde V_i$ by
$\widetilde V_i(\widetilde V_i^*\widetilde V_i)^{-1/2}$ makes it
unitary.  Every component of $X^\circ$ is noncompact and therefore
meets $X^\circ\setminus X_i^\circ$, so the two bundles have equal
rank on every component and $\ch_0(\xi_i)=0$.

For $q\geq1$, write
$\ch_q(\xi_i)=\ch_q(\widetilde B_{1,i})-\ch_q(\widetilde B_{0,i})$.
These forms vanish outside $X_i^\circ$ because $\widetilde V_i$ is
parallel there.  The universal rational identities between Chern
classes and Chern characters, which are also used in
\cite[\S2]{BaerHankeKCowaist}, express a relative Chern number as a
linear combination of their monomial integrals.  Thus there are
positive integers $j_1,\ldots,j_p$ with $j_1+\cdots+j_p=k$ such that
\begin{equation}\label{eq:K-chern-character-monomial}
 \int_{X^\circ}\prod_{\ell=1}^p\ch_{j_\ell}(\xi_i)\ne0.
\end{equation}
For $\mathbf a=(a_1,\ldots,a_p)\in\mathbb R^p$, define
\begin{equation}\label{eq:K-Adams-polynomial}
 P_i(\mathbf a)=
 \int_{X^\circ}
 \left[
 \Ahat(TX^\circ)\wedge
 \prod_{\ell=1}^p
 \left(\sum_{q=1}^k a_\ell^q\ch_q(\xi_i)\right)
 \right]_{[n]}.
\end{equation}
Since each $\ch_q(\xi_i)$ vanishes outside $X_i^\circ$, all the
Chern-character monomials above have compact support. The product $\prod_\ell\ch_{j_\ell}(\xi_i)$ has degree
$2(j_1+\cdots+j_p)=2k=n$.  Only $\Ahat_0=1$ can therefore
contribute to the coefficient of $a_1^{j_1}\cdots a_p^{j_p}$ in
$P_i$, which is \eqref{eq:K-chern-character-monomial}.
The degree of $P_i$ in each variable is at most $k$.  If
$P_i(\mathbf a)=0$ for every
$\mathbf a\in\{1,\ldots,k+1\}^p$, then $P_i$ would vanish as a
polynomial, so the displayed coefficient would be zero, contrary to
\eqref{eq:K-chern-character-monomial}.  Thus we can choose
$\mathbf a\in\{1,\ldots,k+1\}^p$ with $P_i(\mathbf a)\ne0$,
as in the proof of \cite[Lemma~5]{BaerHankeKCowaist}.

Define
\begin{equation}\label{eq:K-relative-top-form}
 \Theta_i=
 \left[
 \Ahat(TX^\circ)
 \prod_{\ell=1}^p
 \left(\sum_{q=1}^k a_\ell^q\ch_q(\xi_i)\right)
 \right]_{[n]}.
\end{equation}
Since the forms $\ch_q(\xi_i)$ vanish outside $X_i^\circ$,
$\supp\Theta_i\subset X_i^\circ$ and
$\int_{X^\circ}\Theta_i=P_i(\mathbf a)\ne0$.
We apply Lemma~\ref{lem:comparison-domains}\textup{(ii)} with
$K_i=X_i^\circ$ and $\Theta_i$.  For all sufficiently large $i$ it
gives $M_i,C_i$ satisfying \ref{K:distance}, with
\[
 X_i^\circ\cap M_i\subset C_i,\qquad
 \int_{M_i}\Theta_i\ne0.
\]

For the chosen integers, we use the construction in the proof of
\cite[Lemma~5]{BaerHankeKCowaist}, equations~(3)--(4).  For
$B=\widetilde B_{1,i}|_{M_i}$ or $\widetilde B_{0,i}|_{M_i}$, it gives
\[
 \Psi_{a_\ell}B=\Psi_{a_\ell}^+(B)-\Psi_{a_\ell}^-(B),
 \qquad
 \ch_q(\Psi_{a_\ell}B)=a_\ell^q\ch_q(B),
 \quad 0\leq q\leq k,
\]
where $\Psi_{a_\ell}^\pm(B)$ are actual bundles obtained from finite
direct sums of tensor products of exterior powers of $B$, with the
induced metrics and connections.  Define bundles over $M_i$ by
\[
\begin{aligned}
 C_{1,\ell}&=\Psi_{a_\ell}^+(\widetilde B_{1,i}|_{M_i})
            \oplus\Psi_{a_\ell}^-(\widetilde B_{0,i}|_{M_i}),\\
 C_{0,\ell}&=\Psi_{a_\ell}^-(\widetilde B_{1,i}|_{M_i})
            \oplus\Psi_{a_\ell}^+(\widetilde B_{0,i}|_{M_i}).
\end{aligned}
\]
Then $\Psi_{a_\ell}(\xi_i|_{M_i})=C_{1,\ell}-C_{0,\ell}$ and
\[
 \ch(C_{1,\ell})-\ch(C_{0,\ell})
 =\left(\sum_{q=1}^k a_\ell^q\ch_q(\xi_i)\right)\Big|_{M_i}.
\]
Define the two bundles by
\[
\begin{aligned}
 A_{1,i}
 &=\bigoplus_{\substack{\epsilon\in\{0,1\}^p\\
              \#\{\ell:\epsilon_\ell=0\}\ \mathrm{even}}}
      \bigotimes_{\ell=1}^p C_{\epsilon_\ell,\ell},\\
 A_{0,i}
 &=\bigoplus_{\substack{\epsilon\in\{0,1\}^p\\
              \#\{\ell:\epsilon_\ell=0\}\ \mathrm{odd}}}
      \bigotimes_{\ell=1}^p C_{\epsilon_\ell,\ell}.
\end{aligned}
\]
Expanding $\bigotimes_{\ell=1}^p(C_{1,\ell}-C_{0,\ell})$ gives
\[
 \ch(A_{1,i})-\ch(A_{0,i})
 =\left[
 \prod_{\ell=1}^p
 \left(\sum_{q=1}^k a_\ell^q\ch_q(\xi_i)\right)
 \right]\Big|_{M_i}.
\]
Every factor on the right has positive degree, so
\[
 \operatorname{rk}(A_{1,i})-\operatorname{rk}(A_{0,i})
 =[\ch(A_{1,i})-\ch(A_{0,i})]_{[0]}=0.
\]

The curvature argument in the proof of
\cite[Lemma~5, equation~(4)]{BaerHankeKCowaist} gives
$\|R^{\Psi_{a_\ell}^\pm(B)}\|_g\leq c_{a_\ell}\|R^B\|_g$,
with $c_{a_\ell}$ depending only on $a_\ell$.  Since
$a_\ell\leq k+1$, $p\leq k$, and
\[
 R^{E\oplus F}=R^E\oplus R^F,\qquad
 R^{E\otimes F}=R^E\otimes1+1\otimes R^F,
\]
there is a constant $b_n>0$, depending only on $n$, such that
\[
 \max_a\|R^{A_{a,i}}\|_g
 \leq b_n\max_a\|R^{\widetilde B_{a,i}|_{M_i}}\|_g
 \leq\frac{b_n}{i}.
\]

On $M_i\setminus C_i\subset X^\circ\setminus X_i^\circ$,
applying the same exterior-power, tensor-product, and direct-sum
operations to $\widetilde V_i|_{M_i\setminus C_i}$ gives parallel
unitary maps.  In particular, define
\[
\begin{aligned}
 W_1:C_{0,1}|_{M_i\setminus C_i}
   &\longrightarrow C_{1,1}|_{M_i\setminus C_i},\\
 W_1(u,v)&=
 \bigl(\Psi_{a_1}^+(\widetilde V_i)v,\,
       \Psi_{a_1}^-(\widetilde V_i^{-1})u\bigr).
\end{aligned}
\]
For every $\epsilon=(\epsilon_1,\ldots,\epsilon_p)$ with an odd number of zero entries, we denote
$\epsilon'=(1-\epsilon_1,\epsilon_2,\ldots,\epsilon_p)$.
The number of zero entries in $\epsilon'$ differs by one and is
therefore even.  Over $M_i\setminus C_i$, define
\[
\begin{aligned}
 V_i^A\big|_{\bigotimes_{\ell=1}^pC_{\epsilon_\ell,\ell}}
 &: \bigotimes_{\ell=1}^pC_{\epsilon_\ell,\ell}
 \longrightarrow
 \bigotimes_{\ell=1}^pC_{\epsilon'_\ell,\ell},\\
 V_i^A(u_1\otimes\cdots\otimes u_p)
 &=\begin{cases}
  W_1u_1\otimes u_2\otimes\cdots\otimes u_p,
       &\epsilon_1=0,\\
  W_1^{-1}u_1\otimes u_2\otimes\cdots\otimes u_p,
       &\epsilon_1=1.
 \end{cases}
\end{aligned}
\]
Changing the first entry again recovers $\epsilon$, so every even
summand occurs exactly once as a target.  Each displayed map is
parallel and unitary. Their direct sum is the parallel unitary bundle
isomorphism
\[
 V_i^A:A_{0,i}|_{M_i\setminus C_i}
       \longrightarrow A_{1,i}|_{M_i\setminus C_i}.
\]
Both connections are flat on this set, which proves \ref{K:exterior}.
Since $\Ahat(TM_i)=\Ahat(TX^\circ)|_{M_i}$,
\[
 \int_{M_i}\Ahat(TM_i)
 \bigl(\ch(A_{1,i})-\ch(A_{0,i})\bigr)
 =\int_{M_i}\Theta_i\ne0.
\]
This proves \ref{K:pairing} and completes the proof.
\end{proof}

For each sufficiently large $i$, let $q_i:Y_i\to M_i$ be the
universal cover and $\Gamma_i=\Deck(Y_i/M_i)$, and set
$E_{a,i}=q_i^*A_{a,i}$.  Since $M_i$ is compact, $\Gamma_i$ is
countable and $q_i$ is a principal $\Gamma_i$-bundle.

\begin{lemma}\label{lem:K-lifted}
The preceding objects have the following properties:
\begin{enumerate}[label=\textup{(K\arabic*)},ref=\textup{(K\arabic*)},
                  start=5,leftmargin=3.4em]
\item\label{K:spin}
$Y_i$ is spin.

\item
Every $\gamma\in\Gamma_i$ admits a unitary spinor map
\[
 \widehat T_{\gamma,i}:S_{Y_i}\longrightarrow S_{Y_i},
 \qquad
 \operatorname{pr}_{S_{Y_i}}\widehat T_{\gamma,i}
 =\gamma\operatorname{pr}_{S_{Y_i}},
\]
preserving chirality, Clifford multiplication, and the spin connection.

\item\label{K:T-maps}
Using the same $\widehat T_{\gamma,i}$ for $a=1,0$, the maps
$T^a_{\gamma,i}=\widehat T_{\gamma,i}\otimes1$
on $S_{Y_i}\otimes E_{a,i}$, where $1$ denotes the canonical map
$(y,v)\mapsto(\gamma y,v)$ on $q_i^*A_{a,i}$, satisfy
Setup~\ref{setup:data}\ref{setup:lifts}.

\item\label{K:exterior-maps}
With
\[
 Z_1=Z_0=Y_i,\qquad
 \Phi=\operatorname{id}_{Y_i},\qquad
 \widehat\Phi=1,\qquad
 U^E=q_i^*V_i^A,
\]
Setup~\ref{setup:data}\ref{setup:exterior}--\ref{setup:compatibility}
holds over $q_i^{-1}(M_i\setminus C_i)$.

\item\label{K:small-contraction}
For all sufficiently large $i$,
\[
 \|\cR^{E_{a,i}}\|<\frac{\sigma}{8},
 \qquad a=1,0.
\]
\end{enumerate}
\end{lemma}

\begin{proof}
The map $Y_i\xrightarrow{q_i}M_i\hookrightarrow X$ lifts, because
$Y_i$ is simply connected, to the universal cover $\widetilde X$ of
$X$.  Pulling back its spin structure gives a spin structure on
$Y_i$.  By \cite[Chapter~II, Theorem~1.7]{LawsonMichelsohn}, this
structure is unique up to isomorphism because
$H^1(Y_i;\mathbb Z_2)=0$.  For $\gamma\in\Gamma_i$, the
orientation-preserving deck isometry $\gamma:Y_i\to Y_i$ pulls back
the spin structure to an isomorphic one.  Thus $\gamma$ admits a
spin lift inducing a unitary spinor map $\widehat T_{\gamma,i}$
preserving chirality, Clifford multiplication, and the spin connection.

The canonical map on $q_i^*A_{a,i}$ preserves the pulled-back
connection because $q_i\circ\gamma=q_i$.
Tensoring it with $\widehat T_{\gamma,i}$ proves
\ref{K:T-maps}.  The map $q_i^*V_i^A$ is unitary and parallel, and,
because the same spin lift is used for both twists,
\[
 T^1_{\gamma,i}(1\otimes q_i^*V_i^A)
 =(1\otimes q_i^*V_i^A)T^0_{\gamma,i}.
\]
This proves \ref{K:exterior-maps}.

Since $q_i$ is a local isometry,
$\|R^{E_{a,i}}\|=\|R^{A_{a,i}}\|_g$.  For a local orthonormal frame
$e_1,\ldots,e_n$, \eqref{eq:twisting-curvature-endomorphism} gives
\[
\begin{aligned}
 \|\cR^{E_{a,i}}\|
 \leq\sum_{r<s}
   \|R^{E_{a,i}}(e_r,e_s)\|
 \leq\binom n2\frac{b_n}{i}<\frac{\sigma}{8}
\end{aligned}
\]
for all sufficiently large $i$.  This proves
\ref{K:small-contraction}.
\end{proof}

\begin{proof}[Proof of Theorem~\ref{thm:main} under
alternative~\ref{alt:K}]
Choose $i$ sufficiently large that
\[
 \dist(C_i,\partial M_i)>\frac{2\pi}{\sqrt\sigma},
 \qquad
 \binom n2\frac{b_n}{i}<\frac{\sigma}{8}.
\]
In Setup~\ref{setup:data}, take
$(M,C,\Gamma,\pi_1,\pi_0)=(M_i,C_i,\Gamma_i,q_i,q_i)$
with the bundles and maps of Lemma~\ref{lem:K-lifted}.
The scalar-curvature bound is inherited from $g$, and
$\ch(E_{a,i})=q_i^*\ch(A_{a,i})$.  Theorem~\ref{thm:vanishing}
therefore gives
\[
 \int_{M_i}\Ahat(TM_i)
 \bigl(\ch(A_{1,i})-\ch(A_{0,i})\bigr)=0,
\]
contrary to Lemma~\ref{lem:K-compact}\ref{K:pairing}.
\end{proof}


\vspace{0.5cm}

\section{Proof under the relative cohomology assumption}\label{sec:H}

Assume alternative~\ref{alt:H} in Theorem~\ref{thm:main} and, for a contradiction, fix a
complete metric $g$ on $X$ with $\Sc_g\geq\sigma>0$.  The relative classes give compactly supported
closed two-forms on suitable domains. After passing to the
universal covers, we use them to construct the twisting data required by Theorem~\ref{thm:vanishing}.

\begin{lemma}\label{lem:H-forms}
For every sufficiently large $i$ there are a connected compact smooth
domain $M_i\subset X^\circ$ with nonempty boundary, a compact smooth
domain $C_i\Subset M_i^\circ$, the universal cover
$q_i:Y_i\to M_i$, and forms
$\omega_i\in\Omega_c^2(M_i)$ and $\alpha_i\in\Omega^1(Y_i)$ with the
following properties:
\begin{enumerate}[label=\textup{(H\arabic*)},ref=\textup{(H\arabic*)},
                  leftmargin=3.4em]
\item\label{H:domains}
$\dist(C_i,\partial M_i)\to\infty$.

\item\label{H:closed-support}
$d\omega_i=0$ and $\supp\omega_i\Subset C_i$.

\item\label{H:pairing}
$\displaystyle\int_{M_i}\omega_i^k\ne0$.

\item\label{H:primitive}
$d\alpha_i=q_i^*\omega_i$.

\item\label{H:exterior-zero}
$\alpha_i=0$ on $q_i^{-1}(M_i\setminus C_i)$.

\item\label{H:lifts}
The map $Y_i\to X^\circ$ lifts to $\widetilde X^\circ$ under
$p^\circ$ and to the spin universal cover of $X$.  In particular,
$Y_i$ is spin.
\end{enumerate}
\end{lemma}

\begin{proof}
We recall the relative de Rham complex of
\cite[\S6, pp.~78--79]{BottTu}, with differential
\[
\begin{gathered}
 \Omega^r(X^\circ,X^\circ\setminus X_i^\circ)
 =\Omega^r(X^\circ)\oplus
 \Omega^{r-1}(X^\circ\setminus X_i^\circ),\\
 d(\xi,\mu)=
 (d\xi,\xi|_{X^\circ\setminus X_i^\circ}-d\mu).
\end{gathered}
\]
Choose a representative $(\xi_i^0,\mu_i^0)$ of $h_i$.  Thus
\[
 d\xi_i^0=0,
 \qquad \xi_i^0=d\mu_i^0
 \quad\text{on }X^\circ\setminus X_i^\circ.
\]
The absolute de Rham map preserves products by
\cite[Theorems~14.28, 15.8 and the discussion following (15.9)]{BottTu}.
Multiplication by a closed form $\eta$ on $X^\circ$ commutes with
the displayed differential and is given by $(\xi,\mu)\eta = \bigl(\xi\wedge\eta,\,\mu\wedge\eta|_{X^\circ\setminus X_i^\circ}\bigr)$. In this complex, $h_i^k$ is the product of $h_i$ with the absolute
class $[\xi_i^0]^{k-1}$ and is represented by
\begin{equation}\label{eq:H-relative-power}
 \bigl((\xi_i^0)^k,
       \mu_i^0\wedge(\xi_i^0)^{k-1}\bigr).
\end{equation}
By the relative de Rham theorem, there is a finite smooth relative
$2k$-cycle $z_i$, with $|\partial z_i|\subset
X^\circ\setminus X_i^\circ$, such that
\[
 \langle h_i^k,z_i\rangle\ne0.
\]

We choose a compact smooth neighborhood $N_i$ of $X_i^\circ$ in
$X^\circ\setminus|\partial z_i|$, contained in the unit neighborhood
of $X_i^\circ$.  The $N_i$ escape every compact subset of $X$.
We then choose $\chi_i\in C^\infty(X^\circ;[0,1])$ with
$\chi_i=0$ near $X_i^\circ$ and
$\supp(1-\chi_i)\Subset N_i^\circ$.  Extend $\chi_i\mu_i^0$ by zero
across $X_i^\circ$ and define
\[
 \xi_i=\xi_i^0-d(\chi_i\mu_i^0).
\]
Then $d\xi_i=0$, $\supp\xi_i\subset\supp(1-\chi_i)$, and
\[
 (\xi_i,(1-\chi_i)\mu_i^0)
 =(\xi_i^0,\mu_i^0)-d(\chi_i\mu_i^0,0)
\]
is still a representative of $h_i$.  Since $\chi_i=1$ near $|\partial z_i|$, we have
\begin{equation}\label{eq:H-cycle-pairing}
 \int_{z_i}\xi_i^k= \int_{z_i}\xi_i^k - \int_{\partial z_i}(1-\chi_i)\mu_i^0\wedge\xi_i^{k-1}= \langle h_i^k,z_i\rangle\ne0.
\end{equation}
Since $|\partial z_i|\cap N_i=\varnothing$, we also have $[\xi_i^k]\ne0$ in $H_c^n(N_i^\circ)$, otherwise a compactly supported
primitive of $\xi_i^k$ in $N_i^\circ$ would, after extension by zero,
give $\int_{z_i}\xi_i^k=0$ by Stokes' theorem. 
By \cite[Corollary~5.8]{BottTu}, for every connected component $W$
of $N_i$, the map
\[
 H_c^n(W^\circ)\longrightarrow\mathbb R,
 \qquad [\eta]\longmapsto\int_{W^\circ}\eta
\]
is an isomorphism.  Hence some component $W_i$ of $N_i$ satisfies
\begin{equation}\label{eq:H-local-form}
 \int_{W_i}\xi_i^k\ne0.
\end{equation}

Since $(p^\circ)^*h_i=0$, the relative differential defined above
applied to a primitive of
$(p^\circ)^*(\xi_i,(1-\chi_i)\mu_i^0)$ gives
\[
 \beta_i\in\Omega^1(\widetilde X^\circ),\qquad
 \nu_i\in C^\infty\!\left(
 \widetilde X^\circ\setminus(p^\circ)^{-1}(X_i^\circ)\right)
\]
with
\[
 d\beta_i=(p^\circ)^*\xi_i,
 \qquad
 \beta_i-d\nu_i=(p^\circ)^*((1-\chi_i)\mu_i^0)
\]
where the second equality holds on
$\widetilde X^\circ\setminus(p^\circ)^{-1}(X_i^\circ)$.
The function $((p^\circ)^*\chi_i)\nu_i$ extends smoothly by zero
across $(p^\circ)^{-1}(X_i^\circ)$, since $\chi_i=0$ near
$X_i^\circ$. The function $\nu_i$ satisfies
$\beta_i=d\nu_i$ outside $(p^\circ)^{-1}\supp(1-\chi_i)$. The
following exact correction therefore makes the primitive vanish
there without changing its differential.   Define
\[
 \widetilde\alpha_i
 =\beta_i-d\bigl(((p^\circ)^*\chi_i)\nu_i\bigr).
\]
Since $\chi_i=1$ on $X^\circ\setminus\supp(1-\chi_i)$,
\begin{equation}\label{eq:H-beta}
 \begin{gathered}
 d\widetilde\alpha_i=d\beta_i=(p^\circ)^*\xi_i,\\
 \widetilde\alpha_i=\beta_i-d\nu_i=0
 \quad\text{on }
 \widetilde X^\circ\setminus(p^\circ)^{-1}\supp(1-\chi_i).
 \end{gathered}
\end{equation}

We apply Lemma~\ref{lem:comparison-domains}\textup{(i)} to $W_i$.  It
gives a connected compact smooth domain $M_i\subset X^\circ$ with
nonempty boundary, $W_i\Subset M_i^\circ$, and
\[
 \dist(W_i,\partial M_i)\longrightarrow\infty.
\]
Let $q_i:Y_i\to M_i$ and $p:\widetilde X\to X$ be the universal
covers.  Since $Y_i$ is simply connected, the compositions of $q_i$
with the inclusions $M_i\hookrightarrow X^\circ$ and
$M_i\hookrightarrow X$ admit lifts
\[
 \begin{aligned}
 \ell_i:Y_i&\longrightarrow\widetilde X^\circ,
 & p^\circ(\ell_i(y))&=q_i(y),\\
 \ell_i^{\mathrm{spin}}:Y_i&\longrightarrow\widetilde X,
 & p(\ell_i^{\mathrm{spin}}(y))&=q_i(y).
 \end{aligned}
\]
The manifold $\widetilde X$ is spin by hypothesis.
We choose a compact smooth domain $C_i\Subset W_i^\circ$ whose interior
contains $W_i\cap\supp(1-\chi_i)$.
Define $\omega_i$ and $\alpha_i$ as the zero extensions to $M_i$
and $Y_i$ of $\xi_i|_{W_i}$ and
$\ell_i^*\widetilde\alpha_i|_{q_i^{-1}(W_i)}$, respectively.
Their support properties make these extensions smooth, and
\[
\begin{gathered}
 d\omega_i=0,
 \qquad \supp\omega_i\Subset C_i,
 \qquad \int_{M_i}\omega_i^k=\int_{W_i}\xi_i^k\ne0,\\
 d\alpha_i=q_i^*\omega_i,
 \quad \alpha_i=0\ \text{on }q_i^{-1}(M_i\setminus C_i).
\end{gathered}
\]
Also
$\dist(C_i,\partial M_i)\geq\dist(W_i,\partial M_i)\to\infty$.
The pullback of the spin structure along $\ell_i^{\mathrm{spin}}$
makes $Y_i$ spin.
\end{proof}

We next use the one-form $\alpha_i$ to equip the trivial Hermitian line bundle
$L_{s,i}=Y_i\times\mathbb C$ with the unitary connection
\[
 \nabla^{s,i}=d-2\pi is\alpha_i,
 \qquad s\in\mathbb R.
\]

\begin{lemma}\label{lem:H-twists}
Let $\Gamma_i=\Deck(Y_i/M_i)$ and fix $y_i\in Y_i$.  Then:
\begin{enumerate}[label=\textup{(H\arabic*)},ref=\textup{(H\arabic*)},
                  start=7,leftmargin=3.4em]
\item\label{H:curvature}
\[
 R^{L_{s,i}}=-2\pi is\,q_i^*\omega_i,
 \qquad
 \|\cR^{L_{s,i}}\|
 \leq 2\pi\sqrt{\frac{n(n-1)}2}\,
       |s|\,\|\omega_i\|_\infty.
\]
Here $\|\omega_i\|_\infty=\sup_{x\in M_i}|\omega_i(x)|_g$,
where $|\omega_i|_g^2=\sum_{a<b}\omega_i(e_a,e_b)^2$ in a
$g$-orthonormal frame.

\item\label{H:chern}
\[
 \ch(L_{s,i},\nabla^{s,i})=q_i^*e^{s\omega_i}.
\]

\item\label{H:potentials}
For every $\gamma\in\Gamma_i$ there is a unique smooth function
$f_{\gamma,i}$ such that
\begin{equation}\label{eq:H-f}
 df_{\gamma,i}=\gamma^*\alpha_i-\alpha_i,
 \qquad f_{\gamma,i}(y_i)=0.
\end{equation}

\item\label{H:maps}
There are unitary spinor bundle maps
\[
 \widehat T_{\gamma,i}:S_{Y_i}\longrightarrow S_{Y_i},
 \qquad
 \operatorname{pr}_{S_{Y_i}}\widehat T_{\gamma,i}
 =\gamma\operatorname{pr}_{S_{Y_i}},
\]
chosen with $\widehat T_{1,i}=1$, for which the bundle maps defined
for $u\otimes z\in(S_{Y_i}\otimes L_{s,i})_y$ by
\[
\begin{aligned}
 T^{s,i}_\gamma:S_{Y_i}\otimes L_{s,i}
 &\longrightarrow S_{Y_i}\otimes L_{s,i},\\
 T^{s,i}_\gamma(u\otimes z)
 &=\widehat T_{\gamma,i}u\otimes
   e^{2\pi isf_{\gamma,i}(y)}z
\end{aligned}
\]
satisfy Setup~\ref{setup:data}\ref{setup:lifts}.

\item\label{H:phase}
On every connected component $N$ of
$q_i^{-1}(M_i\setminus C_i)$, $f_{\gamma,i}$ is constant and
\begin{equation}\label{eq:H-component-phase}
 T^{s,i}_\gamma
 =e^{2\pi isf_{\gamma,i}|_N}T^{0,i}_\gamma
 \qquad\text{on }N.
\end{equation}
\end{enumerate}
\end{lemma}

\begin{proof}
Lemma~\ref{lem:H-forms}\ref{H:primitive} gives
\[
 R^{L_{s,i}}=d(-2\pi is\alpha_i)
 =-2\pi is\,d\alpha_i=-2\pi is\,q_i^*\omega_i.
\]
For a local orthonormal
frame $e_1,\ldots,e_n$, \eqref{eq:twisting-curvature-endomorphism}
and Cauchy--Schwarz give
\[
 \|\cR^{L_{s,i}}\|
 \leq2\pi|s|\sum_{a<b}|\omega_i(e_a,e_b)|
 \leq2\pi\sqrt{\frac{n(n-1)}2}\,|s|\,\|\omega_i\|_\infty.
\]
The normalization in the Introduction gives
\[
 \ch(L_{s,i},\nabla^{s,i})
 =\exp\!\left(\frac{iR^{L_{s,i}}}{2\pi}\right)
 =q_i^*e^{s\omega_i},
\]
which proves \ref{H:chern}.

Since $q_i\circ\gamma=q_i$ and $d\alpha_i=q_i^*\omega_i$,
\[
 d(\gamma^*\alpha_i-\alpha_i)
 =\gamma^*q_i^*\omega_i-q_i^*\omega_i
 =(q_i\circ\gamma)^*\omega_i-q_i^*\omega_i=0.
\]
Thus simple connectedness of $Y_i$ gives a unique solution of
\eqref{eq:H-f} with the stated
normalization.  As in the proof of
Lemma~\ref{lem:K-lifted}, each deck isometry admits a spin lift
inducing a unitary map $\widehat T_{\gamma,i}$
preserving chirality, Clifford multiplication, and the spin connection.
Choose one for each $\gamma$, with $\widehat T_{1,i}=1$.
For a section $v$ of $L_{s,i}$, \eqref{eq:H-f} gives
\[
 (d-2\pi is\,\gamma^*\alpha_i)
       (e^{2\pi isf_{\gamma,i}}v)
 =e^{2\pi isf_{\gamma,i}}
       (d-2\pi is\,\alpha_i)v.
\]
Consequently the maps $T^{s,i}_\gamma$ in \ref{H:maps} are unitary,
preserve the connection, and satisfy \eqref{eq:T-properties}.
Since $q_i\circ\gamma=q_i$, the map $\gamma$ preserves
$q_i^{-1}(M_i\setminus C_i)$.  Lemma~\ref{lem:H-forms}\ref{H:exterior-zero}
and \eqref{eq:H-f} give
\[
 df_{\gamma,i}=\gamma^*\alpha_i-\alpha_i=0
 \quad\text{on }q_i^{-1}(M_i\setminus C_i).
\]
Thus $f_{\gamma,i}$ is constant on each connected component $N$ of
this set, and the formula for $T^{s,i}_\gamma$ gives
\eqref{eq:H-component-phase}.
\end{proof}

Observe that taking $Z_1=Z_0=Y_i$ and $\Phi=\mathrm{id}$ would require $T^{s,i}_\gamma=T^{0,i}_\gamma$
on $q_i^{-1}(M_i\setminus C_i)$. By \eqref{eq:H-component-phase}, these maps instead
differ on each component by a constant phase. We therefore
construct two principal bundles for a common countable group in a way that that the required compatibility holds. The following lemma provides the required data.

\begin{lemma}\label{lem:H-coverings}
Fix $i$ and $s\in\mathbb R$.  There are:
\begin{enumerate}[label=\textup{(H\arabic*)},ref=\textup{(H\arabic*)},
                  start=12,leftmargin=3.4em]
\item\label{H:group}
a countable discrete group $Q_i$;

\item\label{H:covers}
principal $Q_i$-bundles $\pi_{a,i}:Z_{a,i}\to M_i$, $a=1,0$;

\item\label{H:local}
spin structures on $Z_{a,i}$ and Hermitian line bundles
$E_{a,i}\to Z_{a,i}$ with unitary connections such that every
component of $Z_{1,i}$, respectively $Z_{0,i}$, has the metric, spin
structure, and line connection of
$(Y_i,L_{s,i},\nabla^{s,i})$, respectively
$(Y_i,L_{0,i},\nabla^{0,i})$;

\item\label{H:Q-maps}
for every $\theta\in Q_i$, individual smooth unitary bundle maps
\[
 T^a_{\theta,i}:S_{Z_{a,i}}\otimes E_{a,i}
 \longrightarrow S_{Z_{a,i}}\otimes E_{a,i},
 \qquad
 \operatorname{pr}\circ T^a_{\theta,i}=\theta\circ\operatorname{pr},
\]
satisfying \eqref{eq:T-properties}, with $\operatorname{pr}$ denoting
the bundle projection;

\item\label{H:outside-maps}
a $Q_i$-equivariant orientation-preserving isometry
\[
 \Phi_i:\pi_{0,i}^{-1}(M_i\setminus C_i)
 \longrightarrow\pi_{1,i}^{-1}(M_i\setminus C_i),
\]
and, on these inverse images, smooth bundle maps
$\widehat\Phi_i:S_{Z_{0,i}}\to S_{Z_{1,i}}$ and
$U_i^E:E_{0,i}\to E_{1,i}$ satisfying
\[
 \operatorname{pr}_{S_{Z_{1,i}}}\widehat\Phi_i
 =\Phi_i\operatorname{pr}_{S_{Z_{0,i}}},
 \qquad
 \operatorname{pr}_{E_{1,i}}U_i^E
 =\Phi_i\operatorname{pr}_{E_{0,i}},
\]
where $\widehat\Phi_i$ is unitary and preserves chirality, Clifford
multiplication, and the spin connection, and $U_i^E$
is parallel and unitary, such that
\begin{equation}\label{eq:H-Q-equality}
 T^1_{\theta,i}(\widehat\Phi_i\otimes U_i^E)
 =(\widehat\Phi_i\otimes U_i^E)T^0_{\theta,i},
 \qquad \theta\in Q_i;
\end{equation}

\item\label{H:forms-on-base}
\[
 \ch(E_{1,i})=\pi_{1,i}^*e^{s\omega_i},
 \qquad
 \ch(E_{0,i})=\pi_{0,i}^*1.
\]
\end{enumerate}
\end{lemma}
The proof of Lemma~\ref{lem:H-coverings} is given in Appendix~\ref{app:common-cover}.

\vspace{0.5cm}

For every sufficiently large $i$, define
\[
 P_i(s)=\int_{M_i}
 [\Ahat(TM_i)(e^{s\omega_i}-1)]_{[2k]}
 =\sum_{j=1}^k\frac{s^j}{j!}
   \int_{M_i}\Ahat_{2k-2j}(TM_i)\wedge\omega_i^j.
\]
This is a polynomial, and its coefficient of $s^k$ is
\[
 \frac1{k!}\int_{M_i}\omega_i^k\ne0
\]
by Lemma~\ref{lem:H-forms}\ref{H:pairing}.  Hence $P_i(s_i)\ne0$ for all arbitrarily small $s_i>0$.

\begin{proof}[Proof of Theorem~\ref{thm:main} under
alternative~\ref{alt:H}]
Choose $i$ so large that
\[
 \dist(C_i,\partial M_i)>\frac{2\pi}{\sqrt\sigma}.
\]
Choose $s_i\ne0$ with $P_i(s_i)\ne0$ and
\[
 2\pi\sqrt{\frac{n(n-1)}2}\,
 |s_i|\,\|\omega_i\|_\infty<\frac\sigma8.
\]
We apply Lemma~\ref{lem:H-coverings} with $s=s_i$.  Take
\[
 (M,C,\Gamma,\pi_1,\pi_0)
 =(M_i,C_i,Q_i,\pi_{1,i},\pi_{0,i})
\]
in Setup~\ref{setup:data}.  Items~\ref{H:local} and \ref{H:Q-maps}
give the spin structures, line connections, and maps required in
\ref{setup:covering}--\ref{setup:lifts}. The maps in
\ref{H:outside-maps} and \eqref{eq:H-Q-equality} give
\ref{setup:exterior}--\ref{setup:compatibility}.  The metric on every
component is pulled back from $(M_i,g)$, and
Lemma~\ref{lem:H-twists}\ref{H:curvature} gives
\[
 \|\cR^{E_{1,i}}\|<\frac\sigma8,
 \qquad \cR^{E_{0,i}}=0.
\]
Theorem~\ref{thm:vanishing} and \ref{H:forms-on-base} now give
\[
 0=\int_{M_i}\Ahat(TM_i)(e^{s_i\omega_i}-1)=P_i(s_i),
\]
which is a contradiction.  Together with the proof under
alternative~\ref{alt:K}, this proves Theorem~\ref{thm:main}.
\end{proof}


\vspace{0.5cm}

\appendix

\section{Comparison domains and smooth functions}
\label{app:localization}

\begin{lemma}\label{lem:comparison-domains}
Let $g$ be complete on $X$, and let compact sets
$K_i\subset X^\circ$ escape every compact subset of $X$.
\begin{enumerate}[label=\textup{(\roman*)}]
\item If every $K_i$ is connected, there are connected compact smooth
domains $M_i\subset X^\circ$ with nonempty boundary such that
\[
 K_i\Subset M_i^\circ,
 \qquad \dist(K_i,\partial M_i)\longrightarrow\infty.
\]

\item If $\Theta_i\in\Omega_c^n(X^\circ)$ satisfies
$\supp\Theta_i\subset K_i$ and $\int_{X^\circ}\Theta_i\ne0$, there
are connected compact smooth domains $M_i\subset X^\circ$ with
nonempty boundary and compact smooth domains
$C_i\Subset M_i^\circ$ such that
\[
 K_i\cap M_i\subset C_i,
 \qquad \int_{M_i}\Theta_i\ne0,
 \qquad \dist(C_i,\partial M_i)\longrightarrow\infty.
\]
\end{enumerate}
\end{lemma}

\begin{proof}
Set
\[
 r_i=\min\!\left\{i,
       \frac18\dist(K_i,X\setminus X^\circ)\right\},
\]
with $r_i=i$ when $X^\circ=X$.  Then $r_i\to\infty$: otherwise some
$K_i$ would meet a fixed closed neighborhood of the compact set
$X\setminus X^\circ$, which is compact by Hopf--Rinow.  By
\cite[Theorem~1]{AzagraSmooth}, we can choose a smooth function $d_i$ with
\[
 |d_i-\dist(\mathord\cdot,K_i)|<1.
\]
For a regular value $a_i\in(3r_i,4r_i)$, let
$V_i=\{d_i\leq a_i\}$.  The uniform error and Hopf--Rinow show that
$V_i$ is compact. For large $i$ it lies in $X^\circ$, contains $K_i$
in its interior, and
\begin{equation}\label{eq:comparison-distance}
 \dist(K_i,\partial V_i)\geq3r_i-1.
\end{equation}
If $K_i$ is connected, we take the component of $V_i$ containing it.
For the second assertion,
$\int_{V_i}\Theta_i=\int_{X^\circ}\Theta_i\ne0$. We take a component
$M_i$ on which the integral is nonzero and choose $C_i$ to be a
compact smooth neighborhood of $K_i\cap M_i$ contained in its unit
neighborhood.  In either case $\partial M_i\ne\varnothing$, because a
component of $X^\circ$ cannot be compact.  Equation
\eqref{eq:comparison-distance} gives the conclusions, and in the
second case
$\dist(C_i,\partial M_i)\geq3r_i-2$.
\end{proof}

\begin{lemma}\label{lem:smooth-plateau}
Let $M$ be compact with boundary, let $C\Subset M^\circ$ be compact,
and let $0<R<\dist(C,\partial M)$ and $\eta>0$.  There is
$r\in C^\infty(M;[0,R])$ such that
\[
 r=0\ \text{near }C,
 \qquad r=R\ \text{near }\partial M,
 \qquad |dr|\leq1+\eta.
\]
\end{lemma}

\begin{proof}
We attach an outward collar to $M$ and extend the metric to a complete
metric on the resulting boundaryless manifold $\widetilde M$.  Set
\[
 f=\min\{\dist_{\widetilde M}(\mathord\cdot,C),R\}.
\]
This function is $1$-Lipschitz, vanishes on $C$, and equals $R$ on a
neighborhood of $\partial M$, since
$\dist(C,\partial M)>R$ and distance from $C$ is continuous.  Choose
$\epsilon>0$ so small that
\[
 \frac{(R-4\epsilon)(1+\eta)}{1+\epsilon}>R.
\]
We now apply \cite[Theorem~1]{AzagraSmooth}, with approximation and Lipschitz
errors smaller than $\epsilon$, to obtain a smooth $r_0$ with
$|r_0-f|<\epsilon$ and $\operatorname{Lip}(r_0)<1+\epsilon$.
Choose a smooth nondecreasing $\lambda:\mathbb R\to[0,R]$, equal to
$0$ on $(-\infty,2\epsilon]$ and to $R$ on
$[R-2\epsilon,\infty)$, with
\[
 0\leq\lambda'\leq\frac{1+\eta}{1+\epsilon}.
\]
Then $r=(\lambda\circ r_0)|_M$ has the stated properties.
\end{proof}


\vspace{0.5cm}

\section{The covers used in the relative cohomology case}
\label{app:common-cover}

We first enlarge the maps $T^{s_a}_\gamma$ to groups $G_a$ and
use the identifications over the components of $M\setminus C$ to
construct a common countable group $Q$.  We then construct the two
principal $Q$-bundles and the bundle maps required in
Lemma~\ref{lem:H-coverings}.

On a lifted exterior component, the two systems differ by a
constant phase, which may depend on the component.  The group
$\Lambda$ below contains the phases arising from composition and
from the stabilizers of the chosen lifted components. The
construction of $Q$ allows the two systems to be identified along
all these stabilizers.

\begin{proof}[Proof of Lemma~\ref{lem:H-coverings}]
Fix $i$ and $s$. Set $s_1=s$, $s_0=0$, and use $a=1,0$ to index
the two systems. We suppress the index $i$ and write
\[
 \begin{gathered}
 (M,C,Y,q,\omega,\alpha)=(M_i,C_i,Y_i,q_i,\omega_i,\alpha_i),\\
 (L_{s_a},T^{s_a}_\gamma)=(L_{s_a,i},T^{s_a,i}_\gamma),
 \qquad \Gamma=\Deck(Y/M).
 \end{gathered}
\]
Let $N_1,\ldots,N_r$ be the components of $M\setminus C$.
There are finitely many: $M\setminus\operatorname{int}C$ is a compact
smooth manifold with boundary, and removing $\partial C$ does not
change its components. Also $r\geq1$, since $C\Subset M^\circ$ and
$\partial M\ne\varnothing$.

\emph{Proof of \ref{H:group}.}
The spin lifts inducing
$\widehat T_\gamma\widehat T_\eta$ and
$\widehat T_{\gamma\eta}$ induce the same map of oriented
orthonormal-frame bundles.
Since $\ker(\operatorname{Spin}(n)\to SO(n))=\{\pm1\}$ and $Y$ is connected,
\[
 \widehat T_\gamma\widehat T_\eta=\pm\widehat T_{\gamma\eta},
\]
with a sign independent of $y$. Moreover,
\[
\begin{aligned}
 d(f_\eta+f_\gamma\circ\eta-f_{\gamma\eta})
 &=(\eta^*\alpha-\alpha)+\eta^*(\gamma^*\alpha-\alpha)
   -((\gamma\eta)^*\alpha-\alpha)\\
 &=0.
\end{aligned}
\]
Consequently, for $u\otimes v\in(S_Y\otimes L_{s_a})_y$,
\begin{equation}\label{eq:appendix-tau}
\begin{aligned}
 T^{s_a}_\gamma T^{s_a}_\eta(u\otimes v)
 &=\widehat T_\gamma\widehat T_\eta u
   \otimes e^{2\pi i s_a(f_\eta(y)+f_\gamma(\eta y))}v\\
 &=\pm e^{2\pi i s_a(f_\eta(y)+f_\gamma(\eta y)-f_{\gamma\eta}(y))}
   T^{s_a}_{\gamma\eta}(u\otimes v).
\end{aligned}
\end{equation}
The sign is the one displayed above, and the exponential is constant
on $Y$, but need not belong to $\{\pm1\}$.

Choose one component $\widetilde N_j$ of $q^{-1}(N_j)$ and set
\[
 H_j=\{h\in\Gamma:h\widetilde N_j=\widetilde N_j\}.
\]
Equation~\eqref{eq:H-component-phase} reads
\begin{equation}\label{eq:appendix-exterior-phase}
 T_h^s=e^{2\pi isf_h|_{\widetilde N_j}}T_h^0
 \quad\text{on }\widetilde N_j,\qquad h\in H_j.
\end{equation}
Let $\Lambda\subset U(1)$ be the subgroup generated by $-1$ and
the numbers
\[
\begin{gathered}
 e^{2\pi i s_a(f_\eta(y)+f_\gamma(\eta y)-f_{\gamma\eta}(y))}
 \qquad(a=1,0,\ \gamma,\eta\in\Gamma),\\
 e^{2\pi isf_h|_{\widetilde N_j}}
 \qquad(1\leq j\leq r,\ h\in H_j).
\end{gathered}
\]
The first expression is independent of $y$ by the calculation above,
and the second is constant by \eqref{eq:H-component-phase}.
Since $\Gamma\cong\pi_1(M)$ is countable and $r$ is finite,
$\Lambda$ is countable. Regard all groups below as discrete.
For $a=1,0$, define the group of bundle maps
\[
 G_a=\{\lambda T^{s_a}_\gamma:\lambda\in\Lambda,\ \gamma\in\Gamma\},
\]
with multiplication given by composition.
Equation~\eqref{eq:appendix-tau}, including $\eta=\gamma^{-1}$,
proves closure under multiplication and inverses.
Each element has a unique displayed expression, and
\[
 p_a:G_a\longrightarrow\Gamma,\qquad
 p_a(\lambda T^{s_a}_\gamma)=\gamma
\]
is onto with central kernel $\Lambda$, represented by the maps $\lambda I$.

Define
\[
 \phi_j:p_0^{-1}(H_j)\longrightarrow p_1^{-1}(H_j),
 \qquad
 \phi_j(\lambda T_h^0)
 =\lambda e^{-2\pi isf_h|_{\widetilde N_j}}T_h^s.
\]
By \eqref{eq:appendix-exterior-phase},
\[
 \phi_j(\lambda T_h^0)|_{\widetilde N_j}
 =\lambda e^{-2\pi isf_h|_{\widetilde N_j}}
       T_h^s|_{\widetilde N_j}
 =\lambda T_h^0|_{\widetilde N_j}.
\]
Restriction to $\widetilde N_j$ determines a map in either
$p_a^{-1}(H_j)$: its base map determines $h$, and its value in one
fiber determines $\lambda$.
Thus $\phi_j$ preserves products. Its displayed formula is bijective,
so it is an isomorphism.

To identify the corresponding subgroups via $\phi_j$ by conjugation,
while keeping $G_0$ and $G_1$ embedded, we consider the graph of groups with two vertices carrying separate
copies of $G_0$ and $G_1$,  and $r$ edges joining
them.  The group assigned to the $j$-th edge is $p_0^{-1}(H_j)$,
with injective maps
\[
 G_0 \xleftarrow{\;\mathrm{inclusion}\;} p_0^{-1}(H_j)
 \xrightarrow{\;\phi_j\;} G_1.
\]
Choose the first edge as a maximal tree.  By the definition in
\cite[Chapter~I, \S5.1]{SerreTrees}, the fundamental group
$\mathcal L$ has the presentation
\begin{equation}\label{eq:appendix-group}
 \mathcal L=\left\langle G_1,G_0,e_1,\ldots,e_r\ \middle|\
 e_j\phi_j(g)e_j^{-1}=g\ (g\in p_0^{-1}(H_j)),\
 e_1=1\right\rangle.
\end{equation}
Here the $e_j$ are additional generators, and all multiplication
relations within $G_0$ and $G_1$ are included.
By \cite[Chapter~I, \S5.2, Theorem~11 and Corollary~3]{SerreTrees},
the canonical maps $G_0\to\mathcal L$ and $G_1\to\mathcal L$
are injective.
Since $e_1=1$ and $\phi_j(\lambda I)=\lambda I$, their copies of
$\Lambda$ coincide, and
\[
 e_j(\lambda I)e_j^{-1}=\lambda I
 \qquad(\lambda\in\Lambda).
\]
This copy of $\Lambda$ is therefore central in $\mathcal L$. Set
\[
 Q=\mathcal L/\Lambda.
\]
For each $a=1,0$,
\[
 \ker(G_a\longrightarrow Q)=\Lambda,\qquad G_a/\Lambda\cong\Gamma.
\]
Both $G_a$ are countable, so the presentation
\eqref{eq:appendix-group} makes $\mathcal L$ and its quotient $Q$
countable. Take $Q_i=Q$.

\emph{Proof of \ref{H:covers}.}
For $a=1,0$, define $\pi_a:Z_a\to M$ by
\[
 Z_a=(\mathcal L\times Y)/G_a,\qquad
 (\ell g,y)\sim(\ell,p_a(g)y),\qquad
 \pi_a[\ell,y]_a=q(y).
\]
Choose representatives $R_a\subset\mathcal L$ for $\mathcal L/G_a$.
The set $R_a$ is countable because $\mathcal L$ is countable.
Every $\ell\in\mathcal L$ has a unique expression $\ell=rg$ with
$r\in R_a$ and $g\in G_a$, and $[\ell,y]_a=[r,p_a(g)y]_a$. Thus
\[
 R_a\times Y\longrightarrow Z_a,\qquad
 (r,y)\longmapsto[r,y]_a
\]
is a bijection.  We give $Z_a$ the smooth structure of this countable
disjoint union of copies of $Y$.  Since
$\pi_a[r,y]_a=q(y)$, the map $\pi_a$ is a smooth covering of $M$.

Left multiplication on the first factor defines an action of
$\mathcal L$ on $Z_a$, since
\[
 [\ell'\ell g,y]_a=[\ell'\ell,p_a(g)y]_a,
 \qquad g\in G_a.
\]
For $\lambda\in\Lambda$, centrality and $p_a(\lambda I)=1$ give
\[
 [\lambda\ell,y]_a=[\ell\lambda,y]_a=[\ell,y]_a.
\]
Thus the action factors through $Q$.  It is free, since
\[
 [\ell'\ell,y]_a=[\ell,y]_a
 \ \Longrightarrow\
 \ell'\in\ell\ker(p_a)\ell^{-1}=\Lambda
 \ \Longrightarrow\ \ell'\Lambda=1\quad\text{in }Q.
\]
It is transitive on every fiber: if $y'=\gamma y$, the class of
$\ell'T^{s_a}_\gamma\ell^{-1}$ in $Q$ sends $[\ell,y]_a$ to
$[\ell',y']_a$. Thus $\pi_a$ is a principal
$Q$-bundle. The image under $y\mapsto[1,y]_a$ of a relatively compact
measurable fundamental domain for $\Gamma$ on $Y$ is one for $Q$ on
$Z_a$. Take $Z_{a,i}=Z_a$ and $\pi_{a,i}=\pi_a$.

\emph{Proof of \ref{H:local}.}
For each $r\in R_a$, the map $y\mapsto[r,y]_a$ is a diffeomorphism
from $Y$ onto one connected component of $Z_a$. We transport the spin
structure of $Y$ to this component. Let $\alpha_a\in\Omega^1(Z_a)$ be the one-form whose pullback under $y\mapsto[r,y]_a$ is $\alpha$ for every $r\in R_a$.
Define
\[
 E_a=Z_a\times\mathbb C,
 \qquad
 \nabla^{E_a}=d-2\pi i s_a\alpha_a.
\]
On each copy of $Y$, this bundle and connection are
$(L_{s_a},\nabla^{L_{s_a}})$.

The maps $u\otimes v\mapsto[r,u\otimes v]_a$ over $[r,y]_a$ give
a unitary bundle isomorphism
\begin{equation}\label{eq:appendix-full-bundle}
 S_{Z_a}\otimes E_a
 \cong(\mathcal L\times(S_Y\otimes L_{s_a}))/G_a,
 \qquad [\ell g,w]_a=[\ell,gw]_a.
\end{equation}
Indeed, every $\ell\in\mathcal L$ has a unique expression $\ell=rg$
with $r\in R_a$ and $g\in G_a$, and $[\ell,w]_a=[r,gw]_a$.
For $g=\lambda T^{s_a}_\gamma$ and
$u\otimes v\in(S_Y\otimes L_{s_a})_y$, the action is
\begin{equation}\label{eq:appendix-factor}
 g(u\otimes v)
 =\widehat T_\gamma u\otimes
   \lambda e^{2\pi i s_a f_\gamma(y)}v.
\end{equation}
Since these maps preserve the tensor-product connection,
the isomorphism in \eqref{eq:appendix-full-bundle} preserves
connections.

\emph{Proof of \ref{H:Q-maps}.}
Choose one representative $\kappa(\theta)\in\mathcal L$ of each
$\theta\in Q$, using the same choice for $a=1,0$.  Define
\[
 T^a_\theta[\ell,w]_a=[\kappa(\theta)\ell,w]_a.
\]
This is well-defined because, for $g\in G_a$,
$[\kappa(\theta)\ell g,w]_a=[\kappa(\theta)\ell,gw]_a$.
Since $\theta[\ell,y]_a=[\kappa(\theta)\ell,y]_a$,
$\operatorname{pr}\circ T^a_\theta=\theta\circ\operatorname{pr}$.
For every $\ell\in\mathcal L$, the map
$w\mapsto[\ell,w]_a$ identifies $S_Y\otimes L_{s_a}$ with
the restriction of $S_{Z_a}\otimes E_a$ to the component
$\{[\ell,y]_a:y\in Y\}$.  It preserves the Hermitian metric,
chirality, Clifford multiplication, and connection by
\eqref{eq:appendix-factor} and Lemma~\ref{lem:H-twists}.
Under the two maps
\[
 w\longmapsto[\ell,w]_a,
 \qquad
 w\longmapsto[\kappa(\theta)\ell,w]_a,
\]
the map $T_\theta^a$ is the identity on $S_Y\otimes L_{s_a}$.
Thus it is unitary and satisfies \eqref{eq:T-properties}.

\emph{Proof of \ref{H:outside-maps}.}
Every point of $\pi_a^{-1}(N_j)$ has a representative
$[\ell,y]_a$ with $y\in\widetilde N_j$, and two such
representatives differ by an element of $p_a^{-1}(H_j)$.
Define
\begin{equation}\label{eq:appendix-exterior-maps}
\begin{aligned}
 \Phi_j:\pi_0^{-1}(N_j)&\longrightarrow\pi_1^{-1}(N_j),\\
 \Phi_j[\ell,y]_0&=[\ell e_j,y]_1.
\end{aligned}
\end{equation}
For $g\in p_0^{-1}(H_j)$, \eqref{eq:appendix-group} gives
$ge_j=e_j\phi_j(g)$, and
\[
 [\ell ge_j,y]_1
 =[\ell e_j\phi_j(g),y]_1
 =[\ell e_j,p_0(g)y]_1.
\]
Thus $\Phi_j$ is independent of the representative. Its inverse
sends $[\ell,y]_1$ to $[\ell e_j^{-1},y]_0$, and
$\pi_1\Phi_j[\ell,y]_0=q(y)=\pi_0[\ell,y]_0$.
Since $g_{Z_a}=\pi_a^*g$,
\[
 \Phi_j^*g_{Z_1}
 =\Phi_j^*(\pi_1^*g)
 =(\pi_1\circ\Phi_j)^*g
 =\pi_0^*g=g_{Z_0}.
\]
Thus $\Phi_j$ is an isometry. Since $\pi_0$ and $\pi_1$ preserve
orientation and $\pi_1\circ\Phi_j=\pi_0$, it preserves orientation.

For each connected component $D$ of $\pi_0^{-1}(N_j)$, choose
$\ell\in\mathcal L$ with
$D=\{[\ell,y]_0:y\in\widetilde N_j\}$.
The restriction $\Phi_j|_D$ extends to
\[
 [\ell,y]_0\longmapsto[\ell e_j,y]_1,\qquad y\in Y.
\]
This extension is again an orientation-preserving isometry between
the corresponding copies of $Y$.
These two components are simply connected, so their spin
structures are unique up to isomorphism, as in the proof of
Lemma~\ref{lem:H-twists}. Choose a spin lift of this isometry
and restrict it to $D$. For all such $D$, these restrictions define
$\widehat\Phi_j$ on $\pi_0^{-1}(N_j)$, with
$\operatorname{pr}_{S_{Z_1}}\widehat\Phi_j
=\Phi_j\operatorname{pr}_{S_{Z_0}}$.
It is unitary and preserves chirality, Clifford multiplication,
and the spin connection.

To construct $U_j^E$, define
\[
\begin{aligned}
 U_j:(S_{Z_0}\otimes E_0)|_{\pi_0^{-1}(N_j)}
 &\longrightarrow(S_{Z_1}\otimes E_1)|_{\pi_1^{-1}(N_j)},\\
 U_j[\ell,w]_0&=[\ell e_j,w]_1,
\end{aligned}
\]
where $w\in(S_Y\otimes L_0)_y$ and $y\in\widetilde N_j$.
For $g\in p_0^{-1}(H_j)$, the formula for $\phi_j$ gives
\[
 [\ell ge_j,w]_1
 =[\ell e_j\phi_j(g),w]_1
 =[\ell e_j,\phi_j(g)w]_1
 =[\ell e_j,gw]_1,
\]
so $U_j$ is well-defined. Under the maps
$w\mapsto[\ell,w]_0$ and $w\mapsto[\ell e_j,w]_1$
from the proof of \ref{H:Q-maps}, $U_j$ is the identity on
$S_Y\otimes\mathbb C$. Since $\alpha=0$ on $\widetilde N_j$,
both line connections there are $d$. Thus $U_j$ is smooth
and unitary and preserves chirality, Clifford multiplication,
and connection.

After pulling back the target bundles by $\Phi_j$, the map
$(\widehat\Phi_j^{-1}\otimes1)U_j$ takes values in
$S_{Z_0}\otimes\Phi_j^*E_1$ and commutes with every
Clifford multiplication. Since $n$ is even, Clifford products
span all endomorphisms of the complex spinor fiber
\cite[Chapter~I, Theorem~4.3]{LawsonMichelsohn}. Consequently,
for the chosen $\widehat\Phi_j$, there is a unique smooth
unitary line-bundle map $U_j^E$ on $\pi_0^{-1}(N_j)$ such that
$U_j=\widehat\Phi_j\otimes U_j^E$ and
$\operatorname{pr}_{E_1}U_j^E=\Phi_j\operatorname{pr}_{E_0}$.
In particular,
\begin{equation}\label{eq:appendix-tensor-map}
 (\widehat\Phi_j\otimes U_j^E)[\ell,w]_0
 =U_j[\ell,w]_0=[\ell e_j,w]_1.
\end{equation}
Since $U_j$ preserves the tensor-product connection and
$\widehat\Phi_j$ preserves the spin connection, $U_j^E$ is parallel.

Finally, for $\theta\in Q$,
\[
\begin{aligned}
 \Phi_j(\theta[\ell,y]_0)
 =\Phi_j[\kappa(\theta)\ell,y]_0
 =[\kappa(\theta)\ell e_j,y]_1
  =\theta\Phi_j[\ell,y]_0.
\end{aligned}
\]
The definition in \ref{H:Q-maps} and
\eqref{eq:appendix-tensor-map} give
\[
\begin{aligned}
 T^1_\theta U_j[\ell,w]_0
 &=T^1_\theta[\ell e_j,w]_1\\
 &=[\kappa(\theta)\ell e_j,w]_1\\
 &=U_j[\kappa(\theta)\ell,w]_0
  =U_jT^0_\theta[\ell,w]_0.
\end{aligned}
\]
Since $U_j=\widehat\Phi_j\otimes U_j^E$, this is
\[
 T^1_\theta(\widehat\Phi_j\otimes U_j^E)
 =(\widehat\Phi_j\otimes U_j^E)T^0_\theta.
\]
Assembling these maps over $N_1,\ldots,N_r$ gives
$\Phi_i$, $\widehat\Phi_i$, and $U_i^E$, and proves
\eqref{eq:H-Q-equality}.

\emph{Proof of \ref{H:forms-on-base}.}
Every component of $Z_a$ has the original line connection on $Y$, so
\[
 R^{E_a}=-2\pi i s_a\,\pi_a^*\omega,\qquad
 \ch(E_a,\nabla^{E_a})=\pi_a^*e^{s_a\omega}.
\]
For $a=1,0$ these are the two identities in \ref{H:forms-on-base}.
\end{proof}



\vspace{0.5cm}



\begin{thebibliography}{99}

\bibitem{AzagraSmooth}
D. Azagra, J. Ferrera, F. L\'opez-Mesas, and Y. Rangel,
\emph{Smooth approximation of Lipschitz functions on Riemannian manifolds},
J. Math. Anal. Appl. \textbf{326} (2007), no.~2, 1370--1378,

\bibitem{BaerBandaraGeneral}
C. B\"ar and L. Bandara,
\emph{Boundary value problems for general first-order elliptic differential
operators},
J. Funct. Anal. \textbf{282} (2022), no.~12, Paper No.~109445,

\bibitem{BaerBandara}
C. B\"ar and L. Bandara,
\emph{First-order elliptic boundary value problems on manifolds with
noncompact boundary},
Math. Ann. \textbf{393} (2025), 2953--3023,

\bibitem{BourguignonGauduchon}
J.-P. Bourguignon and P. Gauduchon,
\emph{Spineurs, op\'erateurs de Dirac et variations de m\'etriques},
Comm. Math. Phys. \textbf{144} (1992), no.~3, 581--599,

\bibitem{BGV}
N. Berline, E. Getzler, and M. Vergne,
\emph{Heat Kernels and Dirac Operators},
Grundlehren der mathematischen Wissenschaften \textbf{298},
Springer-Verlag, Berlin, 1992; corrected reprint, 2004,

\bibitem{BaerHankeKCowaist}
C. B\"ar and B. Hanke,
\emph{$K$-cowaist of manifolds with boundary},
C. R. Math. Acad. Sci. Paris \textbf{362} (2024), 1349--1356,

\bibitem{BottTu}
R. Bott and L.~W. Tu,
\emph{Differential Forms in Algebraic Topology},
Graduate Texts in Mathematics \textbf{82}, Springer-Verlag,
New York, 1982,

\bibitem{CecchiniLongNeck}
S. Cecchini,
\emph{A long neck principle for Riemannian spin manifolds with positive
scalar curvature},
Geom. Funct. Anal. \textbf{30} (2020), no.~5, 1183--1223,

\bibitem{CecchiniZeidlerCallias}
S. Cecchini and R. Zeidler,
\emph{Scalar curvature and generalized Callias operators},
in M. Gromov and H.~B. Lawson, Jr. (eds.),
\emph{Perspectives in Scalar Curvature}, Vol.~1,
World Scientific, Hackensack, NJ, 2023, pp.~515--542,

\bibitem{CecchiniZeidlerComparison}
S. Cecchini and R. Zeidler,
\emph{Scalar and mean curvature comparison via the Dirac operator},
Geom. Topol. \textbf{28} (2024), no.~3, 1167--1212,

\bibitem{GromovLawsonComplete}
M. Gromov and H.~B. Lawson, Jr.,
\emph{Positive scalar curvature and the Dirac operator on complete
Riemannian manifolds},
Inst. Hautes \`Etudes Sci. Publ. Math. \textbf{58} (1983), 83--196,


\bibitem{GromovMetricInequalities}
M. Gromov,
\emph{Metric inequalities with scalar curvature},
Geom. Funct. Anal. \textbf{28} (2018), no.~3, 645--726,

\bibitem{GromovFourLectures}
M. Gromov,
\emph{Four lectures on scalar curvature},
in M. Gromov and H.~B. Lawson, Jr. (eds.),
\emph{Perspectives in Scalar Curvature}, Vol.~1,
World Scientific, Hackensack, NJ, 2023, pp.~1--514;


\bibitem{GuoXieYuQuantitative}
H. Guo, Z. Xie, and G. Yu,
\emph{Quantitative $K$-theory, positive scalar curvature, and bandwidth},
in M. Gromov and H.~B. Lawson, Jr. (eds.),
\emph{Perspectives in Scalar Curvature}, Vol.~2,
World Scientific, Hackensack, NJ, 2023, pp.~763--798,


\bibitem{HormanderI}
L. H\"ormander,
\emph{The Analysis of Linear Partial Differential Operators I:
Distribution Theory and Fourier Analysis},
Classics in Mathematics, Springer-Verlag, Berlin, 2003,

\bibitem{LawsonMichelsohn}
H.~B. Lawson, Jr. and M.-L. Michelsohn,
\emph{Spin Geometry}, Princeton Mathematical Series 38,
Princeton University Press, Princeton, NJ, 1989.

\bibitem{RoeOpenIndex}
J. Roe,
\emph{An index theorem on open manifolds. I},
J. Differential Geom. \textbf{27} (1988), no.~1, 87--113,

\bibitem{SerreTrees}
J.-P. Serre,
\emph{Trees},
Springer Monographs in Mathematics, corrected second printing,
Springer-Verlag, Berlin, 2003.

\bibitem{ZeidlerBandWidth}
R. Zeidler,
\emph{Band width estimates via the Dirac operator},
J. Differential Geom. \textbf{122} (2022), no.~1, 155--183,


\bibitem{ZhangDeformedDirac}
W. Zhang,
\emph{Deformed Dirac operators and scalar curvature},
in M. Gromov and H.~B. Lawson, Jr. (eds.),
\emph{Perspectives in Scalar Curvature}, Vol.~2,
World Scientific, Hackensack, NJ, 2023, pp.~201--214,

\end{thebibliography}
\end{document}